\documentclass[12pt]{amsart}
\usepackage{amsmath,amsthm,amsfonts,amssymb,latexsym, mathabx}
\usepackage{enumerate}
\usepackage{color}
\usepackage{graphicx,scalerel}
\usepackage{lmodern}

\begin{document}

\theoremstyle{plain}

\newtheorem{thm}{Theorem}[section]
\newtheorem{lem}[thm]{Lemma}
\newtheorem{pro}[thm]{Proposition}
\newtheorem{hyp}[thm]{Hypotheses}
\newtheorem{cor}[thm]{Corollary}
\newtheorem*{conj}{Conjecture}
\newtheorem*{thmA}{Theorem A}

\theoremstyle{definition}
\newtheorem{rem}[thm]{Remark}
\theoremstyle{definition}
\newtheorem{defi}[thm]{Definition}
\theoremstyle{definition}
\newtheorem{ex}[thm]{Example}
\theoremstyle{definition}
\newtheorem{prob}[thm]{Problem}

\newcommand{\Maxn}{\operatorname{Max_{\textbf{N}}}}
\newcommand{\Syl}{\operatorname{Syl}}
\newcommand{\dl}{\operatorname{dl}}
\newcommand{\Con}{\operatorname{Con}}
\newcommand{\cl}{\operatorname{cl}}
\newcommand{\Stab}{\operatorname{Stab}}
\newcommand{\Aut}{\operatorname{Aut}}
\newcommand{\Ker}{\operatorname{Ker}}
\newcommand{\fl}{\operatorname{fl}}
\newcommand{\Irr}{\operatorname{Irr}}
\newcommand{\SL}{\operatorname{SL}}
\newcommand{\FF}{\mathbb{F}}
\newcommand{\NN}{\mathbb{N}}
\newcommand{\N}{\mathbf{N}}
\newcommand{\C}{\mathbf{C}}
\newcommand{\OO}{\mathbf{O}}
\newcommand{\F}{\mathbf{F}}
\newcommand\wh[1]{\hstretch{2}{\hat{\hstretch{.5}{#1}}}}

\renewcommand{\labelenumi}{\upshape (\roman{enumi})}

\newcommand{\PSL}{\operatorname{PSL}}
\newcommand{\PSU}{\operatorname{PSU}}

\providecommand{\V}{\mathrm{V}}
\providecommand{\E}{\mathrm{E}}
\providecommand{\ir}{\mathrm{Irr_{rv}}}
\providecommand{\Irrr}{\mathrm{Irr_{rv}}}
\providecommand{\re}{\mathrm{Re}}

\def\Z{{\mathbb Z}}
\def\C{{\mathbb C}}
\def\Q{{\mathbb Q}}
\def\irr#1{{\rm Irr}(#1)}
\def\irrv#1{{\rm Irr}_{\rm rv}(#1)}
\def \c#1{{\mathcal #1}}
\def\cent#1#2{{\bf C}_{#1}(#2)}
\def\syl#1#2{{\rm Syl}_#1(#2)}
\def\nor{\triangleleft\,}
\def\oh#1#2{{\bf O}_{#1}(#2)}
\def\Oh#1#2{{\bf O}^{#1}(#2)}
\def\zent#1{{\bf Z}(#1)}
\def\det#1{{\rm det}(#1)}
\def\ker#1{{\rm ker}(#1)}
\def\norm#1#2{{\bf N}_{#1}(#2)}
\def\alt#1{{\rm Alt}(#1)}
\def\iitem#1{\goodbreak\par\noindent{\bf #1}}
   \def \mod#1{\, {\rm mod} \, #1 \, }
\def\sbs{\subseteq}

\def\gc{{\bf GC}}
\def\ngc{{non-{\bf GC}}}
\def\ngcs{{non-{\bf GC}$^*$}}
\newcommand{\notd}{{\!\not{|}}}

\def\aut#1{{\rm Aut}(#1)}
\def\inn#1{{\rm Inn}(#1)}
\def\out#1{{\rm Out}(#1)}

\newcommand{\Mult}{{\mathrm {Mult}}}
\newcommand{\Inn}{{\mathrm {Inn}}}
\newcommand{\IBR}{{\mathrm {IBr}}}
\newcommand{\IBRL}{{\mathrm {IBr}}_{\ell}}
\newcommand{\IBRP}{{\mathrm {IBr}}_{p}}
\newcommand{\ord}{{\mathrm {ord}}}
\def\id{\mathop{\mathrm{ id}}\nolimits}
\renewcommand{\Im}{{\mathrm {Im}}}
\newcommand{\Ind}{{\mathrm {Ind}}}
\newcommand{\diag}{{\mathrm {diag}}}
\newcommand{\soc}{{\mathrm {soc}}}
\newcommand{\End}{{\mathrm {End}}}
\newcommand{\sol}{{\mathrm {sol}}}
\newcommand{\Hom}{{\mathrm {Hom}}}
\newcommand{\Mor}{{\mathrm {Mor}}}
\newcommand{\Mat}{{\mathrm {Mat}}}
\def\rank{\mathop{\mathrm{ rank}}\nolimits}
\newcommand{\Tr}{{\mathrm {Tr}}}
\newcommand{\tr}{{\mathrm {tr}}}
\newcommand{\Gal}{{\it Gal}}
\newcommand{\Spec}{{\mathrm {Spec}}}
\newcommand{\ad}{{\mathrm {ad}}}
\newcommand{\Sym}{{\mathrm {Sym}}}
\newcommand{\Char}{{\mathrm {char}}}
\newcommand{\pr}{{\mathrm {pr}}}
\newcommand{\rad}{{\mathrm {rad}}}
\newcommand{\abel}{{\mathrm {abel}}}
\newcommand{\codim}{{\mathrm {codim}}}
\newcommand{\ind}{{\mathrm {ind}}}
\newcommand{\Res}{{\mathrm {Res}}}
\newcommand{\Ann}{{\mathrm {Ann}}}
\newcommand{\Ext}{{\mathrm {Ext}}}
\newcommand{\Alt}{{\mathrm {Alt}}}
\newcommand{\AAA}{{\sf A}}
\newcommand{\SSS}{{\sf S}}
\newcommand{\h}{{\mathcal H}}
\newcommand{\CC}{{\mathbb C}}
\newcommand{\CB}{{\mathbf C}}
\newcommand{\RR}{{\mathbb R}}
\newcommand{\QQ}{{\mathbb Q}}
\newcommand{\ZZ}{{\mathbb Z}}
\newcommand{\NB}{{\mathbf N}}
\newcommand{\OB}{{\mathbf O}}
\newcommand{\ZB}{{\mathbf Z}}
\newcommand{\EE}{{\mathbb E}}
\newcommand{\PP}{{\mathbb P}}
\newcommand{\GC}{{\mathcal G}}
\newcommand{\HC}{{\mathcal H}}
\newcommand{\GA}{{\mathfrak G}}
\newcommand{\TC}{{\mathcal T}}
\newcommand{\SC}{{\mathcal S}}
\newcommand{\RC}{{\mathcal R}}
\newcommand{\GCD}{\GC^{*}}
\newcommand{\TCD}{\TC^{*}}
\newcommand{\FD}{F^{*}}
\newcommand{\GD}{G^{*}}
\newcommand{\HD}{H^{*}}
\newcommand{\GCF}{\GC^{F}}
\newcommand{\TCF}{\TC^{F}}
\newcommand{\PCF}{\PC^{F}}
\newcommand{\GCDF}{(\GC^{*})^{F^{*}}}
\newcommand{\RGTT}{R^{\GC}_{\TC}(\theta)}
\newcommand{\RGTA}{R^{\GC}_{\TC}(1)}
\newcommand{\Om}{\Omega}
\newcommand{\eps}{\epsilon}
\newcommand{\al}{\alpha}
\newcommand{\chis}{\chi_{s}}
\newcommand{\sigmad}{\sigma^{*}}
\newcommand{\PA}{\boldsymbol{\alpha}}
\newcommand{\gam}{\gamma}
\newcommand{\lam}{\lambda}
\newcommand{\la}{\langle}
\newcommand{\ra}{\rangle}
\newcommand{\hs}{\widehat{s}}
\newcommand{\htt}{\widehat{t}}
\newcommand{\tn}{\hspace{0.5mm}^{t}\hspace*{-0.2mm}}
\newcommand{\ta}{\hspace{0.5mm}^{2}\hspace*{-0.2mm}}
\newcommand{\tb}{\hspace{0.5mm}^{3}\hspace*{-0.2mm}}
\def\skipa{\vspace{-1.5mm} & \vspace{-1.5mm} & \vspace{-1.5mm}\\}
\newcommand{\tw}[1]{{}^#1\!}
\renewcommand{\mod}{\bmod \,}

\marginparsep-0.5cm

\renewcommand{\thefootnote}{\fnsymbol{footnote}}
\footnotesep6.5pt

\title{A Reduction Theorem for the Galois--McKay Conjecture}

\author{Gabriel Navarro}
\address{Departament de Matem\`atiques, Universitat de Val\`encia, 46100 Burjassot, Val\`encia, Spain}
\email{gabriel.navarro@uv.es}
\author{Britta Sp\"ath}
\address{BU Wuppertal, Gau\ss str. 20, 42119 Wuppertal, Germany} 
\email{bspaeth@uni-wuppertal.de}
\author{Carolina Vallejo}
\address{Departamento de Matem\'aticas,  
Universidad Aut\'onoma de Madrid, Campus de Cantoblanco, 28049 Madrid, Spain}
\email{carolina.vallejo@uam.es}
\thanks{The research of the first-named author is partially supported by the Prometeo/Generalitat Valenciana,
Ministerio
de Econom\'ia y Competividad MTM2016-76196-P and FEDER Funds.
The research  of the second-named author is also supported by the research training group
\emph{GRK 2240: Algebro-Geometric Methods in Algebra, Arithmetic and Topology}, funded by the DFG.
The third-named author acknowledges support by MTM2016-76196-P  and the ICMAT Severo Ochoa
project SEV-2011-0087}



\begin{abstract}
We introduce $\c H$-triples and a partial order relation on them, generalizing the theory of ordering character triples developed by Navarro and Sp\"ath. This generalization takes into account the action of Galois automorphisms on characters and, together with previous results of Ladisch and Turull, allows us to reduce the Galois--McKay conjecture to a question about simple groups.
 \end{abstract}

\maketitle
\bigskip

\def\C{{\Bbb C}}
\def\irr#1{{\rm Irr}(#1)}
\def\irrp#1{{\rm Irr}_{p'}(#1)}
\def\irrpr#1{{\rm Irr}_{p'}^{{\rm rel}}(#1)}
\def\irrpq#1{{\rm Irr}_{p', \Q_{p'}}(#1)}
\def\irrpa#1{{\rm Irr}_{p', \mathcal A}(#1)}
\def\cent#1#2{{\bf C}_{#1}(#2)}
\def\syl#1#2{{\rm Syl}_#1(#2)}
\def\nor{\triangleleft\,}
\def\zent#1{{\bf Z}(#1)}
\def\norm#1#2{{\bf N}_{#1}(#2)}
\def\oh#1#2{{\bf O}_{#1}(#2)}
\def\iitem#1{\goodbreak\par\noindent{\bf #1}}

\def\irrp#1{{\rm Irr}_{p'}(#1)}

\section*{Introduction}
The origin of the McKay conjecture dates back to a paper of John McKay from 1972 (\cite{McK72}), 
where it is stated for finite simple groups and for $p=2$.
\begin{conj}[The McKay conjecture]
Let $G$ be a finite group, let $p$ a prime and let $H=\norm GP$ be the normalizer in $G$
of a Sylow $p$-subgroup $P$
of $G$. Then $$|{\rm Irr}_{p'}(G)|=|{\rm Irr}_{p'}(H)| \, ,$$ where 
${\rm Irr}_{p'}(G)$ is the set of irreducible complex characters of $G$
of degree not divisible by $p$. 
\end{conj}
In 2007, Martin Isaacs, 
Gunter Malle and the first-named author reduced the McKay conjecture to a problem
about simple groups in \cite{IMN}. 
Using this reduction theorem, G. Malle and the 
second-named author have recently proven that the McKay conjecture holds
for all finite groups for $p=2$ in \cite{MS16}. 

\medskip

In 2004, the first-named author predicted
that not only the degrees of the complex
characters of $G$ and $H$ were related but also their values (see \cite{Nav04}).
For a fixed prime $p$, let $\c H$ be the subgroup of $\mathcal G={\rm Gal}( \Q^{\rm ab}/\Q)$
  consisting of the $\sigma \in \c G$ for which there exists an integer $f$  such that
  $\sigma(\xi)=\xi^{p^f}$ for  every root of unity $\xi$ of order not divisible by $p$. 
  
\begin{conj}[The Galois--McKay conjecture]
Let $G$ be a finite group, let $p$ be a prime, and let $H=\norm GP$ be the normalizer in $G$
of a Sylow $p$-subgroup $P$
of $G$. 
Then the actions of $\c H$ on $\irrp G$ and $\irrp H$ 
are permutation isomorphic.
\end{conj}

As a matter of fact,  this conjecture is stated only for cyclic
subgroups of $\c H$ in Conjecture A of \cite{Nav04},
but it is suggested in the above more general form at the end of the same paper.
The Galois--McKay conjecture as 
stated above is equivalent to the existence of a McKay bijection
preserving fields of values of characters over the field $\Q_p$ of $p$-adic numbers. 
Recall that if $\Q \sbs \mathbb F$ is a field extension and $\chi$ is a character
of a group $G$, then the field of values $\mathbb F (\chi)$ of $\chi$ over $\mathbb F$ is 
obtained by adjoining to $\mathbb F$ all the values of $\chi$. The conjecture appeared
in this latter form in \cite{T08} (also including local Schur indices).

\medskip

The  Galois--McKay conjecture has been proved for
$p$-solvable groups in \cite{T1} and for alternating groups in \cite{Nath} and \cite{BN}. 
 It has been established 
 for groups with cyclic Sylow $p$-subgroups  in \cite{Nav04};
and  for groups of Lie type in  
 defining characteristic in \cite{Ruh}.  For sporadic groups,
it can now be easily checked with \cite{GAP}. 
 
Also, some of its main consequences
have been obtained since its formulation.
For instance, in \cite{NTT07} it was proven that, for $p$ odd, $\norm GP=P$
if, and only if, $G$ has no non-trivial $p$-rational valued irreducible
character of $p'$-degree. 
More recently, for $p=2$, 
it has been proved in
\cite{SF18} that $\norm GP=P$ if, and only if, all the odd-degree irreducible
characters of $G$ are fixed
by $\sigma_0 \in \c H$, where $\sigma_0$ squares odd roots of unity and fixes 2-power roots of unity. Some other consequences, 
such as determining the exponent of 
 $P/P'$ from the character table have been treated 
recently in \cite{NT18}. In particular, we now know 
that the character table determines the exponent of the 
abelianization of a Sylow 2-subgroup thanks to \cite{NT18} and \cite{M19}.
 In all these papers,
ad-hoc reductions to simple groups
 have been provided for fixed elements $\sigma \in \c H$, and then the
 classification of finite simple groups has been used 
 to prove the theorems. However,
 the Galois--McKay conjecture has eluded a
 general reduction until now.
 The following is the main result of this paper.
 We recall that a simple group $S$ is {\bf involved} in $G$ if 
 $S\cong K/N$ for some $N\nor K \leq G$.
 
 \begin{thmA}
  Suppose that $G$ is a finite group, and $p$ is a prime.
 If all simple groups involved in $G$ satisfy the inductive 
 Galois--McKay condition for $p$ (Definition \ref{inductive}),
 then the Galois--McKay conjecture holds for $G$ and $p$.
 \end{thmA}

One of the main differences between our reduction theorem and the 
reduction theorem for the McKay conjecture is that we cannot make use of the 
general theory of character triples and character triple isomorphisms, since
these do not preserve in general fields of values.
We remedy this by introducing the notion of $\c H$-triples in Section \ref{Htriples}.
There we also introduce a partial order relation between $\c H$-triples that allows us 
to construct $\c H$-equivariant bijections between character sets. The original
partial order relation between character triples that we now generalize and whose use
is crucial in this work was introduced
in \cite{NS14}. Here we mostly refer to the exposition given in  \cite{Nav18}.
 In Section 
\ref{newfromold} we study how to construct new ordered $\c H$-triple pairs 
from old ones. In Section \ref{inductivecondition} we give the inductive 
Galois--McKay condition that we expect all finite simple groups to satisfy. 
Finally in Section \ref{thereduction}, we prove Theorem A relying on 
key results due to Friedrich Ladisch and Alexandre Turull. 
We care to remark that our Theorem A does not 
provide a different proof of the $p$-solvable case of the
Galois--McKay conjecture as our method depends on the
study of the character theory over Glauberman
correspondents that has been carried out by A. Turull in different papers.

\medskip

The verification of the inductive Galois--McKay condition 
for finite simple groups
brings up a new challenge, as it requires a vast knowledge of the 
character values of decorated simple groups and the interplay between Galois action 
and the action of group automorphisms on characters, a subject that it is still 
not fully understood.
Examples of families of simple groups satisfying the inductive
Galois--McKay condition will appear in \cite{Spa19}. 

\medskip

\noindent {\bf Acknowledgements.}~~
The authors are strongly indebted to Gunter Malle for an exhaustive 
revision of a previous version of this paper. This material is partially based upon work supported by the 
National Science Foundation under Grant No. DMS-1440140 while the authors were in residence 
at the Mathematical Sciences Research Institute in Berkeley, California, during the Spring 2018 semester.
  They would like to thank the MSRI and the staff for the kind hospitality. 
The third-named author is obliged to Gus Lonergan for 
useful conversations during the afore-mentioned program.

  \section{$\c H$-triples}\label{Htriples}
  
  Let $G$ be a finite group, let $N\nor G$, and let $\theta \in \irr N$.  
We denote by $\irr{G|\theta}$ the set of  $\chi \in \irr G$
such that $\theta$ is an irreducible constituent of the restriction $\chi_N$.
If $\theta$ is $G$-invariant, then it is said that $(G,N,\theta)$ is a {\bf character triple}.
The aim of this section is to extend the theory of ordering character triples
developed in \cite{NS14} by taking into account the action of
 Galois automorphisms on characters.

\medskip

  Let ${\c G}={\rm Gal}( \Q^{\rm ab}/\Q)$, where
  $\Q^{\rm ab}$ is the  field generated by all roots of unity in $\C$. 
  By Brauer's theorem on splitting fields \cite{Bra45}, the group
  $\c G$ acts on the irreducible characters of every finite group. 
 Let $\sigma \in \c G$, we denote by $\theta^\sigma$ the irreducible character of $N$
given by $\theta^\sigma(n)=\theta(n)^\sigma=\sigma(\theta(n))$ for every $n \in N$. 
(Note that $\c G$ is abelian.)
  
  \smallskip
  
  Let $p$ be a prime
  which is fixed but arbitrary.
  Let $\c H$ be the subgroup of $\c G$
  consisting of the $\sigma \in \c G$ for which there exists an integer $f$  such that
  $\sigma(\xi)=\xi^{p^f}$ for  every root of unity $\xi$ of order not divisible by $p$. 
  For every non-negative integer $n$,
   the restriction of the automorphisms in $\c H$ to $\QQ(\xi_n)$ yields a group 
   $\c H_n\leq {\rm Gal}(\QQ(\xi_n)/\QQ)$ which is 
   isomorphic to 
   ${\rm Gal}(\QQ_p(\xi_n)/\QQ_p)$, where $\xi_n$ is a primitive $n$th root of 
   unity and $\QQ_p$ is the field of $p$-adic numbers.
    
  \smallskip

We denote by $\theta^{\c H}$ the $\c H$-orbit of $\theta $  and by 
$\irr{G|\theta^{\c H}}$ the set of irreducible characters of $G$ which 
lie over some $\c H$-conjugate of $\theta$.  
This set is $$\bigcup_{\sigma \in \c H} \irr{G|\theta^\sigma}\, .$$
If $\chi \in \irr{G|\theta^{\c H}}$, then we call the natural number $\chi(1)/\theta(1)$
the {\bf character degree ratio} of $\chi$ (with respect to $\theta^{\c H}$).

\smallskip

 We denote by $G_{\theta^\c H}$ the stabilizer in $G$ of the set $\theta^{\c H}$,
 in the action of $G$ on $\irr N$ by conjugation. 
 Note that $G_{\theta^{\c H}}=\{ g \in G \ | \ \theta^g=\theta^\sigma \text{ for some }\sigma \in \c H\}$.
  We write $(G,N,\theta)_{\c H}$  if $G_{\theta^{\c H}}=G$;
  in other words, if
   $$ \{ \theta^g \mid g \in G\}  \sbs \{ \theta^\sigma \mid \sigma \in \mathcal H\} .$$
   In this case, we call $(G,N,\theta)_{\c H}$ an {\bf $\c H$-triple}.
   Notice that if $(G,N,\theta)_{\c H}$ is an $\c H$-triple,
  then $(G,N,\theta^\sigma)_{\c H}$ is also an $\c H$-triple for every $\sigma \in \c H$.
  Also, note that $(G_{\theta^\c H},N, \theta)_{\c H}$ is always an $\c H$-triple.

  \smallskip
  
  Suppose that $(G,N,\theta)_{\c H}$
  is an $\c H$-triple. Let $G_\theta$ be the stabilizer of $\theta$ in $G$.
  If $g \in G$, then there is $\sigma \in \c H$
  such that $\theta^g=\theta^\sigma$. Therefore $(G_\theta)^g=G_\theta$,
  and we have that $G_\theta \nor G$.
  Furthermore,
  notice that via
  $gG_\theta \mapsto \sigma\c H_\theta$
  we obtain an injective homomorphism $G/G_\theta \rightarrow \c H/\c H_\theta$.
  We will denote by $\c H_{G, \theta}$ the subgroup of $\c H$ such that 
  $\c H_{G, \theta}/\c H_\theta$ is the image under the above monomorphism. 
 We will write just $\c H_G$ whenever $\theta$ is clear from the context.
 
We start with the following result on projective representations. 
For a background on these,
see Chapter 11 of \cite{Is} or  Section 10.4 of \cite{Nav18}.
 It is a version of the main result of
 \cite{Re65}.
  
      \begin{thm}\label{projectivecyclotomic2}
  Suppose that $\c Q$ is a projective representation of $G$
   whose factor set $\alpha$ only takes roots of unity values and satisfies $\alpha(1,1)=1$.
   Then there is a similar representation $\c Q'$ with 
   entries in a finite cyclotomic extension of $\QQ$.
  \end{thm}
  
  \begin{proof}
As in Theorem 5.6 of \cite{Nav18}, let $Z$ be a finite subgroup of 
$\C^\times$ containing all values of $\alpha$. Define $\widehat G=\{ (g, z) \ | \ g \in G, \ z \in Z\}$ 
with multiplication given by 
$$(g_1, z_1)(g_2, z_2)=(g_1g_2, \alpha(g_1, g_2)z_1z_2).$$
The product above is associative as $\alpha$ is a factor set. Since 
$\alpha(g,1)=\alpha(1,1)=\alpha(1,g)=1$ for every $g \in G$ 
(by Lemma 11.5 of \cite{Is}), we get that $\alpha(g, g^{-1})=\alpha(g^{-1}, g)$,
and that $\hat G$ is a finite group.  Define $\widehat{\c Q}((g,z))=z\c Q (g)$ 
for every $(g,z)\in \widehat G$. Then $\widehat {\c Q}$ is an ordinary 
representation of $\widehat G$.  
By Brauer's theorem (Theorem 10.3 of \cite{Is}), there exists a 
representation $\widehat{\c D}$ of $\widehat G$ similar to $\widehat{\c Q}$ 
with matrix entries in some finite cyclotomic extension of $\QQ$.
Then
we easily check that
 $\c Q'(g)=\widehat{\c D}(g,1)$ is a projective representation of $G$ similar to $\c Q$ 
 and such the values of $\c Q'$ lie
 in some finite cyclotomic extension of $\Q$. 
  \end{proof}
  
  Given a character triple $(G,N,\theta)$, one can find a projective representation $\c P$ of $G$
   {\bf associated} with $\theta$ in the sense of Definition 5.2 of \cite{Nav18}.

  \begin{cor}\label{projectivecyclotomic}
  If $(G,N,\theta)$ is a character triple,
  then there is
  a projective representation
   $\c P$ of $G$ associated with $\theta$
   with entries in $\QQ^{\rm ab}$ and whose
   factor set only takes roots of unity values.
   If $\c P$ is any such representation, then
   $\c P(g)$ has finite order for every $g \in G$.
\end{cor}

\begin{proof}
By Theorem 5.5 of \cite{Nav18}, there exists a projective
representation $\c P'$ of $G$ associated with $(G,N,\theta)$
such that the factor set $\alpha$ only takes roots of unity values.
Since $\alpha$ is the factor set of $\c P'$, we see that $\alpha(1,1)=1$.
By Theorem \ref{projectivecyclotomic2}, let $\c P$ be
a similar projective representation of $G$ with values 
in $\Q^{ab}$. Since $\c P$ and $\c P'$ have the
same factor set, it easily follows that $\c P$ is a projective
representation associated with $\theta$ satisfying
the required properties.

Finally, if $\c P$ is any such representation, let $Z$ be a finite subgroup of 
$\C^\times$ containing all values of $\alpha$. Define $\widehat G=\{ (g, z) \ | \ g \in G, \ z \in Z\}$ 
with multiplication given by 
$$(g_1, z_1)(g_2, z_2)=(g_1g_2, \alpha(g_1, g_2)z_1z_2)$$
as in the proof of Theorem \ref{projectivecyclotomic2}. 
Then $\widehat {\c P}$ defined as $\widehat {\c P }(g, z)=z \c P (g)$ for every $(g,z )\in \widehat G$ is an ordinary representation
of $\widehat G$. The last statement follows since $\widehat{\c P}(g,z)$ has finite order for every $(g, z )\in \widehat G$.
\end{proof}
    
 \begin{rem}\label{uniquenessmu} Recall that if
  $(G,N, \theta)$ is a character triple, then
  two projective representations $\c P$ and $\c P_1$ of $G$
  are associated with $\theta$ if, and only if, there
  is a function $\mu\colon G \rightarrow \CC^\times$ constant on cosets
  of $N$, with $\mu(1)=1$ and a complex invertible matrix $M$
  such that $\c P_1(g)=\mu(g) M^{-1}\c P(g)M$ for all $g \in G$
  (this is Lemma 10.10(b) of \cite{Nav18}).
  Notice that $\mu$ is uniquely determined by the pair $(\c P, \c P_1)$.
  Indeed, if
  $\c P_1(g)=\mu_1(g) M_1^{-1}\c P(g)M_1$ for all $g \in G$, for some
  function $\mu_1$ with $\mu_1(n)=1$ for all $n \in N$, 
  then we will have that $M^{-1}\c P(n) M=M_1^{-1}\c P(n) M_1$
  for all $n \in N$. By Schur's lemma, 
  we have that $M_1=\lambda M$ for some $\lambda \in \CC^\times$.
  Hence $\mu(g)M^{-1}\c P(g) M=\mu_1(g) M^{-1}\c P(g) M$ for
  all $g \in G$, and thus $\mu(g)=\mu_1(g)$ using that $M^{-1}\c P(g) M$
  is non-zero.
  \end{rem}
  
  Let $\c P$ be a projective representation of $G$ with factor set $\alpha$. 
  If $f \colon G \rightarrow G_1$
  is a group isomorphism (we will use exponential notation for images of $f$), then 
  we define $\c P^f (g_1)=\c P (g_1^{f^{-1}})$
  for $g_1 \in G_1$. This is a projective representation of $G_1$
  with factor set $\alpha^f$, where $\alpha^f(x^f,y^f)=\alpha(x, y)$
  for $x, y \in G$.
  (See Theorem 10.9 of \cite{Nav18}.) If $\sigma \in {\rm Gal}(\Q^{\rm ab}/\Q)$,
  and $\c P$ has entries in $\Q^{\rm ab}$, then
  $\c P^\sigma(g):=\c P(g)^\sigma$ (where $\sigma$ is applied entry-wise) defines a projective
  representation of $G$ with factor set $\alpha^\sigma(x,y)=\alpha(x,y)^\sigma$, for
  $x, y \in G$. Notice that $\alpha(x,y) \in \Q^{\rm ab}$ in this case.
  
  In what follows, we shall use that if $N \nor G$, then $G \times \c G$
  acts naturally on $\irr N$. Indeed, if $\theta \in \irr N$, $g \in G$ and
  $\sigma \in \c G$, then $\theta^{g\sigma}$ is the irreducible
  character of $N$ given by $\theta^{g\sigma}(n)=\theta(gng^{-1})^\sigma$ for $n \in N$.
  (To simplify notation, we often use $g\sigma$ instead of $(g,\sigma)$.)
  
  \smallskip
  
   Throughout this work, we will use $\sim$
  to denote similarity between matrices.

    \begin{lem}\label{mu}
  Suppose that $N \nor G$, $\theta \in \irr N$, and assume that
  $\theta^{g\sigma }=\theta$ for some $g\in G$ and
  $\sigma \in {\rm Gal}(\QQ^{\rm ab}/\QQ)$.
  Let $\c P$ be a projective representation of $G_\theta$
  associated with $\theta$ with values in $\QQ^{\rm ab}$ and factor set $\alpha$.
  Then $$\c P^{g\sigma}(x)=\c P(gxg^{-1})^\sigma,$$ where we apply $\sigma$ 
  entry-wise,  defines
  a projective representation of $G_\theta$ 
  associated with $\theta$, with factor set $\alpha^{g\sigma}(x, y)=\alpha^g(x,y)^\sigma$.
  In particular, there is a unique function $$\mu_{g\sigma} \colon G_\theta \rightarrow \CC^\times$$
  with $\mu_{g\sigma}(1)=1$, constant on cosets of $N$
  such that $\c P^{g\sigma} \sim \mu_{g\sigma} \c P$.
   \end{lem}
   
   \begin{proof}
   Notice that $g$ normalizes $G_\theta$. The rest is straightforward 
   using Remark \ref{uniquenessmu}. \end{proof}

  \medskip
  We are now ready to define a partial order relation between $\c H$-triples. 
   
  \begin{defi}\label{Hiso}
    Suppose that $(G,N,\theta)_{\c H}$ and $(H,M,\varphi)_{\c H}$
  are $\c H$-triples.
  We write
  $(G,N,\theta)_{\c H} \ge_c(H,M,\varphi)_{\c H}$
  if the following conditions hold:
  \begin{enumerate}[(a)]
  \item[(i)]
  $G=NH$, $N\cap H=M$, $\cent {G}N \sbs H$.
  
  \item[(ii)]
  $(H \times \c H)_\theta=(H \times \c H)_\varphi$.
  In particular, $H_\theta=H_\varphi$.
  
  \item[(iii)]
  There are projective representations $\c P$ of $G_\theta$
  and $\c P'$ of $H_\varphi$  associated with $\theta$ and $\varphi$ 
  with entries in $\QQ^{\rm ab}$
  with factor sets $\alpha$ and $\alpha'$ respectively such that 
   $\alpha$ and $\alpha'$ take roots of unity values,
  $\alpha_{H_\theta \times H_\theta}=\alpha'_{H_\theta \times H_\theta}$, 
and
  for $c \in \cent GN$, the scalar matrices $\c P(c)$ and $\c P'(c)$ are 
  associated with the same
  scalar $\zeta_c$.
  
  \item[(iv)] For every $a \in (H\times \c H)_\theta$, the functions $\mu_a$
  and $\mu'_a$  given by Lemma \ref{mu} agree on $H_\theta$.

  \end{enumerate}
  \end{defi}
  
  In (iii), notice that if $c \in \cent GN$, then $c \in H_\theta$,
  and $\c P(c)$ and $\c P'(c)$ are scalar matrices by Schur's Lemma
  (applied to the irreducible representations $\c P_N$ and $\c P'_M$).
  
  In the situation described above we 
  say that $(\c P, \c P')$  {\bf gives} 
  $$(G,N,\theta)_{\c H} \ge_c(H,M,\varphi)_{\c H}\, .$$
  
  \medskip

 Note that if $(\c P, \c P')$ gives
 $(G,N,\theta)_{\c H} \ge_c(H,M,\varphi)_{\c H}$ as above, then $(\c P, \c P')$ is associated 
 with $(G_\theta, N,\theta) \ge_c(H_\varphi,M,\varphi)$ in the sense of 
 Definition 10.14 of \cite{Nav18}.  
 
 \medskip
 
 The following technical result will be useful at the end of Section \ref{newfromold}.

\begin{lem}\label{replace}
Assume that $(\c P, \c P)$ gives $(G, N, \theta)_{\c H}\geq_c (H, M, N)_{\c H}$. 
Let $U\leq \C^\times$ be the subgroup of roots of unity of $\C$.
If $\epsilon \colon G_\theta \rightarrow U$ is any map constant on $N$-cosets
and such that $\epsilon (1)=1$, then
$(\epsilon \c P, \epsilon_{H_\theta} \c P')$ also gives 
$(G, N, \theta)_{\c H}\geq_c (H, M, \varphi)_{\c H}$.
\end{lem}
\begin{proof} Conditions (i) and (ii) of Definition \ref{Hiso} are satisfied. 
Write $\hat{\c P}=\epsilon \c P$ and 
$\hat{\c P'}=\epsilon_{H_\theta} \c P'$, we will show that $(\hat{\c P}, \hat{\c P'})$ gives $(G, N, \theta)_{\c H}\geq_c(H, M, \varphi)_{\c H}$.  

Note that
$\hat{\c P}$ is a projective representation with values in $\Q^{ab}$ 
associated with $\theta$.
If $\nu \colon G_\theta \rightarrow U$ is any arbitrary function, we can define
$\delta(\nu)\colon G_\theta \times G_\theta \rightarrow U$ by
$$\delta(\nu)(x, y)=\nu(x)\nu(y)\nu(xy)^{-1}\, $$
so that $\delta(\nu)$ is a factor set. It is routine to check that the factor set of $\widehat{\c P}$ is
$\beta=\delta(\epsilon)\alpha$.
Also, $\hat{\c P'}$ is a projective representation with values 
in $\Q^{ab}$ associated with $\varphi$, and with
factor set $\beta'=\delta(\epsilon_{H_\theta})\alpha'=\beta_{H_\theta \times H_\theta}$. 
For every $c \in \cent G N$, the matrices $\hat{\c P'}(c)$ and $\hat{\c P'}(c)$
correspond to the same scalar $\epsilon(c)\zeta_c$, where $\c P (c)$ and $\c P'(c)$
correspond to the scalar $\zeta_c$. Hence $(\hat{\c P}, \hat{\c P'})$ 
satisfies condition (iii) of Definition \ref{Hiso}.
Whenever $(h, \sigma)\in (H\times \c H)_\theta$, it is straightforward to check that, 

$$\hat{\c P}^{h\sigma}\sim \hat \mu_{h\sigma}\hat{\c P} \text{ \   \ and \  \ }(\hat{\c P'})^{h\sigma}\sim \hat \mu'_{h\sigma}\hat{\c P'} \, ,$$

where $\hat \mu_{h\sigma}=\mu_{h\sigma}\frac{\epsilon^{h\sigma}}{\epsilon}$,
$\hat \mu'_{h\sigma}=\mu'_{h\sigma}(\frac{\epsilon^{h\sigma}}{\epsilon})_{H_\theta}$,
and the functions $\mu_{h\sigma}$ and $\mu'_{h\sigma}$ given by Lemma \ref{mu} satisfy

 $$\c P ^{h\sigma} \sim \mu_{h\sigma}\c P \text{ \ \  and \ \ } (\c P ')^{h\sigma}\sim \mu'_{h\sigma} \c P'\, .$$
 
 In particular $\hat\mu'_{h \sigma}=(\hat \mu_{h\sigma})_{H_\theta}$, 
 thus the pair $(\hat{\c P}, \hat{\c P'})$ satisfies
 all the conditions in Definition \ref{Hiso}.
\end{proof}

The following technical lemma allows us to show that in order to check (iv) of Definition \ref{Hiso} on $(H\times \c H)_\theta$ it is enough to
check the condition on a 
transversal of $H_\theta$ in $(H\times \c H)_\theta$.

\begin{lem}\label{conjugationstabilizer} Suppose that $(G,N,\theta)_{\c H}$
is an $\c H$-triple. Let $\c P$ be a projective
representation of $G_\theta$ 
associated with $\theta$ with entries in $\Q^{\rm ab}$, with factor set $\alpha$. Then the following hold:
\begin{enumerate}
\item[{\rm (a)}] Let $g \in G_\theta$. Then $\c P ^g(y)=\mu_g(y)M \c P (y) M^{-1}$ for all $y \in G_\theta$, where 
$$\mu_g(y)=\frac{\alpha(g,g^{-1})}{\alpha(g, yg^{-1})\alpha(y, g^{-1})} $$
and $M=\c P(g)$.  In particular, $\mu_g$ has values in $\Q^{\rm ab}\setminus\{0\}$.
\item[{\rm (b)}] Let $(g, \sigma) \in (G\times \c H)_{\theta}$ and suppose that we write
 $g=tx$ where $t \in G_\theta$ and $x \in G$.  Then $\theta^{x \sigma}=\theta$,
 $(G_\theta)^x=G_\theta$,  and $$\mu_{g \sigma}=\mu_t^{x \sigma}\mu_{x\sigma} \, ,$$ 
 where $\mu_t^{x \sigma}(y)=\mu_t(xyx^{-1})^\sigma$, for every $y \in G_\theta$, and the functions $\mu_t$, $\mu_{x\sigma}$
 and $\mu_{g\sigma}$
 are given by Lemma \ref{mu}.
\end{enumerate}
\end{lem}
\begin{proof} Part (a) follows directly
 from the definitions of   $\c P^g$, of a projective
 representation, the uniqueness in Remark \ref{uniquenessmu}
 and Lemma \ref{mu} applied to $\sigma=1$. Given $(g, \sigma) \in (G\times \c H)_{\theta}$, suppose
 that we write $g=tx$ for some $t \in G_\theta$ and $x \in G$. Note that $\theta^{x \sigma}=\theta$. 
 In particular, $x$ normalizes $G_\theta$. Then for every $y \in G_\theta$ we have that
\begin{align*}
\c P ^{g\sigma}(y)& = \c P (gyg^{-1})^\sigma\\[1.5mm]
& = \c P (tx y x^{-1}t^{-1})^\sigma\\[1.5mm]
&=(\c P ^t(x y x^{-1}))^\sigma\\[1.5mm]
& \sim \mu_{t }(x y x^{-1})^\sigma \c P(x y x^{-1})^\sigma\\[1.5mm]
& =  \mu_{t }^{x \sigma}( y) \c P^{x \sigma}( y ) \\[1.5mm]
& \sim \mu_{t }^{x \sigma}( y)\mu_{x\sigma }(y)  \c P (y).  \qedhere
\end{align*}
\end{proof}

We will often use the following fact in Section \ref{newfromold}.

\begin{lem}\label{cor_trans}
 Suppose that $(G,N,\theta)_{\c H}$ and $(H,M,\varphi)_{\c H}$
are $\c H$-triples satisfying the  conditions (i) and (ii) in Definition \ref{Hiso}. 
Write $A= (H\times \c H)_\theta$.
Suppose that $\c P$ and $\c P'$ are projective representations satisfying  (iii) 
from Definition \ref{Hiso}. Then condition (iv) of  Definition \ref{Hiso} holds for every $a\in A$  
if, and only if, it holds for a complete set of representatives of $H_\theta$-cosets in $A$. \end{lem}
\begin{proof}
Note that $H_\theta \nor A$ since $\theta^{h\sigma}=\theta$ implies
that $h$ normalizes $H_\theta$. The direct implication trivially holds. Assume that (iv) of 
Definition \ref{Hiso} holds for a complete set of representatives $\mathbb T$ of 
the $H_\theta$-cosets in $A$. Given $a \in A$, write $a=hx\sigma$ for $h \in H_\theta$ 
and $x\sigma \in \mathbb T$. By Lemma \ref{conjugationstabilizer}(b) we have that
$$\mu_a=\mu_h^{x\sigma}\mu_{x \sigma} \text{ \ and \ } \mu'_a=(\mu'_h)^{x\sigma}\mu'_{x \sigma}.$$
By assumption $\mu_{x \sigma}'$ is the restriction of $\mu_{x \sigma}$. 
By Lemma \ref{conjugationstabilizer}(a) we have that $\mu_h$ depends only 
on the factor set $\alpha$ of $\c P$ and $\mu'_h$ depends only on the 
factor set $\alpha'$ of $\c P'$. Since $\c P$ and $\c P'$ satisfy condition (iii) 
of Definition \ref{Hiso}, we have that $\alpha'$ is the restriction of $\alpha$. 
Hence $\mu'_a$ is the restriction of $\mu_a$ as wanted. 
\end{proof}

Let $(\c P, \c P')$ be associated with $(G,N,\theta)_{\c H} \ge_c(H,M,\varphi)_{\c H}$. 
Since  $(\c P, \c P')$ is associated with $(G_\theta,N,\theta) \ge_c(H_\varphi,M,\varphi)$ 
as in Definition 10.14 of \cite{Nav18}, we have defined character bijections  via $(\c P, \c P')$
$$\tau_J \colon \irr{J|\theta} \rightarrow \irr{J\cap H|\varphi}$$
whenever $N\sbs J\leq G_\theta$ as in Theorem 10.13 of \cite{Nav18}.
These bijections preserve ratios of character degrees.

   \begin{lem}\label{sigma_equiv}
   Suppose that $(\c P, \c P')$ is associated with $(G,N,\theta)_{\c H} \ge_c(H,M,\varphi)_{\c H}$.
   
   \begin{enumerate}[{\rm (a)}]
   \item
   For every $N \sbs J \leq G_\theta$, let 
    $\tau_J\colon \irr{J|\theta} \rightarrow \irr{J\cap H|\varphi}$
     be the bijective map defined via $(\c P, \c P')$. 
    If
     $(h, \sigma) \in (H \times \c H)_\theta$, then $J^h \sbs G_\theta$ and 
     $$\tau_{J^h}(\chi^{h\sigma})=\tau_J(\chi)^{h\sigma}$$
     for every $\chi \in \irr{J|\theta}$.
     
    \item
     Let $\c H_G /\c H_\theta\leq \c H/\c H_\theta$ be the image of
     $G/G_\theta$ in $\c H/\c H_\theta$ under the natural monomorphism.
    If for $\chi \in \irr{G|\theta}$,   we define $\tau(\chi)=(\tau_{G_\theta}(\psi))^H$, 
    where $\psi \in \irr{G_\theta}$ is the Clifford correspondent of $\chi$ lying over $\theta$, 
    then the map $\tau\colon \irr{G|\theta}\rightarrow \irr{H|\varphi}$ is an 
    $\c H_G$-equivariant bijection  preserving ratios of character degrees.
     \end{enumerate}
\end{lem}

\begin{proof}
We have that $(\c P, \c P')$
 is associated with $(G_\theta, N, \theta)\geq_c(H_\varphi, M, \varphi)$ in 
 the sense of Definition 10.14 of \cite{Nav18}. For every $N\sbs J \leq G_\theta$, 
 we have defined character bijections
  $$\tau_J\colon \irr{J|\theta}\rightarrow \irr{J\cap H |\varphi}.$$  
  Given $\chi \in \irr{J|\theta}$, recall that $\chi$ is the trace of a 
  representation of the form $\c Q \otimes
  \c P_J$, where $\c Q$ is an irreducible projective representation of 
  $J/N$, with factor set $\beta=(\alpha^{-1})_{J\times J}$, that can be 
  chosen with matrix entries in some finite cyclotomic extension of $\Q$ 
  (by Theorem \ref{projectivecyclotomic2}). Then $\tau_J(\chi)$ is the character afforded by 
  $ \c Q_{J\cap H} \otimes \c P'_{J\cap H}$. 

\smallskip

Let $a=(h, \sigma)\in A=(H\times \c H)_\theta$. Then $\theta^a=\theta$ 
and also $\varphi^a =\varphi$, as $A=(H\times \c H)_\varphi.$ 
Hence $\chi^a \in \irr{J^h|\theta}$. The character $\chi^a$ 
of $J^h$
is afforded by 
$$(\c Q \otimes \c P_J)^a=\c Q^a \otimes
(\c P_J)^a=
\c Q^a \otimes
(\c P^a)_{J^h} \sim  \c Q ^a \otimes (\mu_ a)_{J^h}\c P_{J^h}=(\mu_a)_{J^h}\c Q ^a  \otimes \c P _{J^h}\, .$$ 
This implies that $(\mu_a)_{J^h}\c Q ^a$ is a projective representation of $J^h/N$
with factor set $(\alpha^{-1})_{J^h \times J^h}$.
By definition, we have that 
$\tau_J(\chi^a)$ is afforded by $$(\mu_ a \c Q^a)_{J^h \cap H} \otimes
\c P '_{J^h \cap H} =(\c Q ^a)_{J^h \cap H}  \otimes
(\mu'_a\c P')_{J^h \cap H} \sim (\c Q_{J\cap H})^a \otimes (\c P'_{J\cap H})^a  \, .$$ 
Just notice that 
$(\c Q_{J\cap H})^a \otimes (\c P'_{J\cap H})^a=(\c Q_{J\cap H} \otimes \c P'_{J\cap H} )^a$ affords $\tau_J(\chi)^a$.

\smallskip

We next prove the second statement. If $\sigma \in \c H_G$, then let 
$g \in H$ be such that $\theta^{g\sigma}=\theta$ (there exists such 
$g \in H$ by the definition of $\c H_G$, and using the fact
that $G=G_\theta H$). Since $(g, \sigma)\in (H\times \c H)_\theta=(H\times \c H)_\varphi$, we have 
$\varphi^{g\sigma}=\varphi$. Consequently $\chi^\sigma \in \irr{G|\theta}$ and 
$\tau(\chi)^\sigma\in \irr{H|\varphi}$.  Since $\psi^{g\sigma}$ is the 
Clifford correspondent of $\chi^\sigma$, we have that 
$$\tau(\chi^\sigma)=\tau_{G_\theta}(\psi^{g\sigma})^H=(\tau_{G_\theta}(\psi)^{g\sigma})^H
 = (\tau_{G_\theta}(\psi)^H)^\sigma=\tau(\chi)^\sigma\, ,$$
where $\tau_{G_\theta}(\psi^{g\sigma})=\tau_{G_\theta}(\psi)^{g\sigma}$ 
by the first part of this proof, so $\tau$ is $\c H_G$-equivariant.
Notice that our map
preserves ratios of character degrees because character triple isomorphisms do,
and it is a bijection since $\tau_{G_\theta}$ and the Clifford correspondence
are bijections. 
\end{proof}
  
  Ordered pairs of $\c H$-triples yield $\c H$-equivariant
  bijections between 
   related character sets. 
  
  \begin{thm}\label{Hequivariant}
  Suppose that $(G,N,\theta)_{\c H} \ge_c(H,M,\varphi)_{\c H}$.
  Then there is an $\c H$-equivariant bijection
  $\irr{G|\theta^{\c H}} \rightarrow \irr{H|\varphi^{\c H}}$ that preserves ratios of character degrees.
   \end{thm}
   
   \begin{proof}
   By the definition of $\c H$-character triples, we have that
   $G$ acts on $\theta^{\c H}$. 
   Suppose that $\theta_1, \ldots, \theta_s$ are representatives
   of the $G$-orbits,
   where say $\theta_1=\theta$, and $\theta_i=\theta^{\sigma_i}$.
   Notice that $G_{\theta_i}=G_\theta$.
   Also notice that if we set $\varphi_i=\varphi^{\sigma_i}$,
   then $\varphi_1, \ldots, \varphi_s$ are representatives
   of the $H$-orbits on $\varphi^{\c H}$ (using that $(H\times \c H)_\theta=(H\times \c H)_\varphi )$.
   If $(\c P, \c P')$ is  associated with
   $(G,N,\theta)_{\c H} \ge_c(H,M,\varphi)_{\c H}$,
   then $(\c P^{\sigma_i}, (\c P')^{\sigma_i})$ is associated with 
   $(G,N,\theta_i)_{\c H} \ge_c(H,M,\varphi_i)_{\c H}$.
   We have that
   $$\irr{G|\theta^{\c H}}=\dot{\bigcup_{i}} \, \irr{G|\theta_i} \text{  \ and \  } \irr{H|\varphi^{\c H}}=\dot{\bigcup_{i}} \, \irr{H|\varphi_i}.$$
  Note that if we write $\tau_i$ for the bijection 
  $\irr{G|\theta_i}\rightarrow \irr{H|\varphi_i}$ given by Lemma \ref{sigma_equiv}, then 
  $\tau_i\circ\sigma_i=\sigma_i\circ \tau_1$. Let $\c H_G$ be as in Lemma \ref{sigma_equiv}. 
  Recall $\c H_G/\c H_\theta\leq \c H /\c H_\theta$ is isomorphic to $G/G_\theta$. 
  Note that $\c H_G$ is the stabilizer in $\c H$ of any of the $G$-orbits on 
  $\theta^{\c H}$ (this is because $\c H$ is abelian). Hence $\c H=\dot \bigcup_i \c H_G \sigma_i$.   
 Let $\tau \colon \irr{G|\theta^{\c H}}\rightarrow \irr{H|\varphi^{\c H}}$ 
 be the bijection defined in the obvious way from the bijections $\tau_i$. 
 Given $\chi \in \irr{G|\theta_1}$ and $\sigma=\omega \sigma_i \in \c H$ with $\omega \in \c H_G$,
  we have that 
$$\tau(\chi^\sigma)=\tau_i(\chi^{\omega \sigma_i})=
\tau_i(\chi^{\sigma_i})^\omega=
(\tau_1(\chi))^{\sigma_i \omega}=
\tau(\chi)^\sigma,$$
 where we use that the bijections $\tau_i$ are $\c H_G$-equivariant and 
 $\tau_i\circ\sigma_i=\sigma_i\circ \tau_1$. It easily follows that $\tau$ is $\c H$-equivariant. 
As each $\tau_i$ preserves character degree ratios, then so does $\tau$.
  \end{proof}
   
 Let $N \nor G$ and $\theta \in \irr N$ not necessarily satisfying $G_{\theta^{\c H}}=G$.
 By the Clifford correspondence, induction of characters defines a bijection
 $$\irr{X|\theta^{\sigma}}\rightarrow \irr{G|\theta^{\sigma}}\, ,$$
 for every $\sigma \in \c H$, whenever $G_\theta \sbs X\leq G$.
 Hence induction of characters defines an $\c H$-equivariant surjective
 map
 $$\irr{X|\theta^{\c H}}\rightarrow \irr{G|\theta^{\c H}}\, ,$$
 which turns out to be injective if, and only if, $G_{\theta^{\c H}}\sbs X$.
 
 \begin{cor}\label{corHequivariant}
 Let $N \nor G$ and $H\leq G$ be such that $G=NH$. Write $M=N\cap H$. Suppose
 that $(G_{\theta^{\c H}}, N, \theta)_{\c H}\geq_c (H_{\varphi^{\c H}}, M, \varphi)_{\c H}$. 
 Then there is an $\c H$-equivariant bijection
 $$\irr{G|\theta^{\c H}}\rightarrow \irr{H|\varphi^{\c H}}$$
 that preserves ratios of character degrees. 
 \end{cor}
 \begin{proof}
 Let $\tau\colon \irr{G_{\theta^{\c H}}|\theta^{\c H}}\rightarrow \irr{H_{\varphi^{\c H}}|\varphi^{\c H}}$
 be the $\c H$-equivariant bijection preserving ratios of character degrees given by Theorem \ref{Hequivariant}. 
Define $\hat \tau \colon \irr{G|\theta^{\c H}}\rightarrow \irr{H|\varphi^{\c H}} $ in the following way.
For $\chi \in \irr{G|\theta^{\c H}}$, let $\psi\in \irr{G_{\theta^{\c H}}|\theta^{\c H}}$ be such that
$\psi^G=\chi$, then $\hat \tau(\chi):=\tau(\psi)^H$. The conclusion then follows from the comments
preceding this result. 
 \end{proof}

\section{Constructing new $\c H$-triples from old ones}\label{newfromold}
The following results show 
easy ways to construct $\c H$-triples from given ones.
The proofs are straightforward from the definitions. 
Note that if $(G,N,\theta)_{\c H} \geq_c(H,M,\varphi)_{\c H}$ and 
$N\sbs J\leq G$, then
 $$(J,N,\theta)_{\c H} \ge_c(J\cap H,M,\varphi)_{\c H}.$$

\begin{lem}\label{iso} Let $(G,N,\theta)_{\c H} \geq_c(H,M,\varphi)_{\c H}$. 
Suppose that $f \colon G \rightarrow \hat G$ is a group isomorphism. 
Then  $(\hat G,\hat N,\theta^f)_{\c H} \ge_c(\hat H,\hat M,\varphi^f)_{\c H}$
where we write $\hat J =J^f$ (using exponential notation for images of $f$) 
and $\psi^f(x^f)=\psi(x)$ for every $\psi \in \irr J$ and $x \in J$.
\end{lem}

Sometimes it will be easier to apply the following 
weaker version of the above result. 
\begin{lem}\label{conjugate} Let $N\nor G$ and $H\leq G$ be such 
that $G=NH$. Write $M=H\cap N$. Suppose that $N\sbs K\leq G$ 
and $(K, N, \theta)_{\c H}\geq_c(K\cap H, M, \varphi)_{\c H}$. 
Then for every $h \in H$
$$(K^h, N, \theta^h)_{\c H}\geq_c(K^h\cap H, M, \varphi^h)_{\c H}.$$
\end{lem}

Under the hypotheses of the above lemma, assume that $(\c P, \c P')$ gives 
$(K, N, \theta)_{\c H}\geq_c(K\cap H, M, \varphi)_{\c H}$ and let  $h \in H$, then
we will consider that 
$$(K^h, N, \theta^h)_{\c H}\geq_c(K^h\cap H, M, \varphi^h)_{\c H}$$
is given by $(\c P^h, (\c P')^h)$. Hence, if $\tau_\theta$ and 
$\tau_{\theta^h}$ are the bijections given by Theorem \ref{Hequivariant},
we will have that $\tau_{\theta^h}(\chi^h)=\tau_\theta(\chi)^h$
for every $\chi \in \irr{G|\theta^{\c H}}$. In particular, if
$\hat \tau_{\theta}$ and $\hat \tau_{\theta^h}$ are the bijections 
given by Corollary \ref{corHequivariant}, then $\hat \tau_{\theta}=\hat \tau_{\theta^h}$.

\begin{lem}\label{Hconjugate}
Let $(G,N,\theta)_{\c H} \geq_c(H,M,\varphi)_{\c H}$ and $\sigma \in \c H$.
Then $(G,N,\theta^\sigma)_{\c H} \geq_c(H,M,\varphi^\sigma)_{\c H}$.
\end{lem}

Suppose that $(\c P, \c P')$ gives $(G, N, \theta)_{\c H}\geq_c (H, M, \varphi)_{\c H}$. 
For every $\sigma \in \c H$, 
 we will always consider that 
 $(G, N, \theta^\sigma)_{\c H}\geq_c (H, M, \varphi^\sigma)_{\c H}$
 is given by $(\c P^\sigma, (\c P')^\sigma)$. In this way, if
 $\tau_\theta$ and $\tau_{\theta^\sigma}$ are the bijections given by Theorem \ref{Hequivariant}, then by construction
 $\tau_{\theta^\sigma}(\chi^\sigma)=\tau_\theta(\chi)^\sigma$ for every $\chi \in \irr{G|\theta^{\c H}}$.
 By Theorem \ref{Hequivariant} $\tau_\theta$ is $\c H$-equivariant, 
 in particular $\tau_{\theta^\sigma}=\tau_\theta$. Hence, by construction
 the bijections given by Corollary \ref{corHequivariant} are also equal.

\begin{lem}\label{quotient} Let $(G,N,\theta)_{\c H} \geq_c(H,M,\varphi)_{\c H}$. 
Suppose that $L\nor G$ is contained in $\ker \theta \cap \ker \varphi\cap \cent G N$, 
and $\cent {G/L}{N/L}=\cent G N /L$. Then
$$(G/L,N/L,\theta)_{\c H} \geq_c(H/L,M/L,\varphi)_{\c H},$$ where 
$\theta$ and $\varphi$ are considered as characters of $N/L$ and $M/L$.
\end{lem}

\begin{lem}\label{direct}
Let $(G_i, N_i, \theta_i)_{\c H} \geq_c(H_i, M_i, \varphi_i)_{\c H}$ for $i=1,2$. Then
$$(G_{\theta^{\c H}}, N, \theta)_{\c H}\geq_c (H_{\varphi^{\c H}}, M, \varphi)_{\c H},$$
where $G=G_1\times G_2$, $H=H_1\times H_2$, $N=N_1\times N_2$, 
$M=M_1\times M_2$, $\theta=\theta_1\times \theta_2$ and $\varphi=\varphi_1\times \varphi_2$.
\end{lem}
\begin{proof}
The group theoretical conditions are easily checked. 
Also, it is easy to check that $(H_{\theta^\c H} \times \c H)_\theta=
(H_{\varphi^\c H} \times \c H)_\varphi$. Now we have to construct
appropriate   projective representations of $(G_{\theta^\c H})_\theta=G_\theta=(G_1)_{\theta_1} \times
(G_2)_{\theta_2}$ and 
$(H_{\varphi^\c H})_\varphi=H_\varphi=(H_1)_{\varphi_1} \times
(H_2)_{\varphi_2}$. This is done
as in Lemma 10.20  of \cite{Nav18}. 
 Checking conditions
(ii), (iii) and (iv) of Definition \ref{Hiso} is straightforward.
\end{proof}

Denote by ${\sf S}_m$ the symmetric group acting on $m$ letters. 
In the next two results we deal with $\c H$-triples and wreath products of groups.
We follow the notation in Chapter 10 of \cite{Nav18}. If $G$ is a finite group, then $G^m$
will denote the direct product $G \times \cdots \times G$ ($m$ times), and if $\theta \in \irr G$,
then in our context $\theta^m=\theta \times \cdots \times \theta \in \irr{G^m}$. Recall that 
${\sf S}_m$ acts naturally on $G^m$ by
$$(g_1, \ldots, g_m)^\omega=(g_{\omega^{-1}(1)}, \ldots, g_{\omega^{-1}(m)})$$
whenever $g_i \in G$ and $\omega \in {\sf S}_m$.

\begin{lem}\label{wreath} Let $(G, N, \theta)_{\c H}$ and 
$(H, M, \varphi)_{\c H}$ be $\c H$-triples such that $(G, N, \theta)_{\c H}\geq_c (H, M, \varphi)_{\c H}$. 
For any $m\in \Z_{>0}$
$$( (G_\theta \wr {\sf S}_m)\Delta^m G, N^m, \theta^m)_{\c H}\geq_c ((H_\theta \wr {\sf S}_m)\Delta^m H, M^m, \varphi^m)_{\c H},$$ 
where $\Delta^m\colon G \rightarrow G^m$ denotes the diagonal embedding of $G$ into the direct product $G^m$.\end{lem}
\begin{proof}
We claim that $(G^m)_{(\theta^m)^{\c H}}= (G_\theta)^m \Delta^m G$.
 Let $(g_1, \ldots, g_m)\in ( G^m)_{(\theta^m)^{\c H}}$. 
 Then there exists some $\sigma \in \c H$ such that 
 $( \theta^m)^{(g_1, \ldots, g_m)}=( \theta^m )^{\sigma}$.
  Hence $\theta^{g_i}=\theta^\sigma$ for each $i$. 
In particular, $G_\theta g_i=G_\theta g_j$
for every $i$ and $j$.  Hence, 
we can write $g_i=x_ig_1$ for some  
 $x_i \in G_\theta$ and for every $i$. 
 Thus $(g_1, \ldots, g_m) \in (G_\theta)^m \Delta^m G$.
The other inclusion is also clear using that given $g \in G$, there
is $\sigma \in \c H$ such that $\theta^g=\theta^\sigma$, by our hypothesis. 
This also implies that
$$(G\wr {\sf S}_m)_{{(\theta^m)}^{\c H}}=(G_\theta \wr {\sf S}_m)\Delta^m G\, .$$
Similarly, $(H\wr {\sf S}_m)_{{(\varphi^m)}^{\c H}}=(H_\varphi \wr {\sf S}_m)\Delta^m H=(H_\theta \wr {\sf S}_m)\Delta^m H$ 
as $H_\theta=H_\varphi$ by hypothesis. 

\smallskip

We follow the proof of Theorem 10.21 of \cite{Nav18}.
First, we easily check that $(G \wr {\sf S}_m)_{\theta^m}=G_\theta  \wr {\sf S}_m$.
Conditions (i) and (ii) of Definition \ref{Hiso} for 
$$( (G_\theta \wr {\sf S}_m)\Delta^m G, N^m, \theta^m)_{\c H}\geq_c ((H_\theta \wr {\sf S}_m)\Delta^m H, M^m, \varphi^m)_{\c H}$$ 
follow from the above discussion together with the discussions inTheorem 10.21 of \cite{Nav18}. 
Let $(\c P, \c P')$ be associated with  $(G, N, \theta)_{\c H}\geq_c (H, M, \varphi)_{\c H}$. 
Construct projective representations $\tilde{\c P}$ and 
$\tilde{\c P'}$ of $G_\theta \wr {\sf S}_m$ and of $H_\theta \wr {\sf S}_m$ as in Theorem 10.21 of \cite{Nav18}. 
Condition (iii) of Definition \ref{Hiso} is proven in Theorem 10.21 \cite{Nav18}. 
It remains to check condition (iv) of Definition \ref{Hiso}.

\smallskip

For any $(\gamma,\sigma) \in
 ((H \wr {\sf S}_m)\times \c H)_{\theta^m}$, we have that 
 $\gamma \in (H_\theta \wr {\sf S}_m)\Delta^m H$. We denote by 
 $\tilde \mu_{\gamma \sigma}$ and $\tilde \mu'_{\gamma \sigma}$ 
 the functions given by Lemma \ref{mu} with respect to 
 the action of $\gamma \sigma$ on 
 $\tilde{\c P}$ 
 and $\tilde{\c P'}$. By Lemma \ref{cor_trans} we only need to check 
 the condition for a transversal of $H_\theta\wr{\sf S}_m$ in 
 $((H_\theta \wr {\sf S}_m)\Delta^m H \times \c H)_{\theta^m}$. 
 In particular, it is enough to check the condition for elements 
 $(\gamma, \sigma)$ such that $\gamma=(y, \ldots, y)=\Delta^m y$ 
 for some $y \in H$ with $\theta^{y\sigma}=\theta$.
 
We check below that, for every $x_i \in H_\theta$ and $\omega \in {\sf S}_m$,
 $$\tilde \mu _{ \gamma \sigma}((x_1, \ldots, x_m)\omega)=\prod_{i=1}^m \mu_{y\sigma}(x_i).$$ 
 Given $(x_1,\ldots, x_m)\omega \in H_\theta\wr {\sf S}_m$ we have that
 \begin{align*}
 \tilde{\c P}^{ \gamma \sigma}((x_1,\ldots, x_m)\omega)& = \tilde{\c P}((x_1^{y^{-1}},\ldots, x_m^{y^{-1}})\omega)^\sigma\\[1.7mm]
 &=(\c P^{y\sigma}(x_1)\otimes \cdots \otimes \c P^{y\sigma}(x_m))\c X _{\theta(1)} (\omega)\\
 &\sim \prod_{i=1}^m \mu_{y \sigma}(x_i) \tilde{\c P}((x_1,\ldots,x_m)\omega), 
 \end{align*}
 where the permutation representation $\c X_{\theta(1)}$ is as in Theorem 10.21 of \cite{Nav18}. 
 We have used that 
 $\c P ^{y\sigma}=\mu_{y \sigma}M \c P M^{-1}$ 
 and 
 $(M\otimes \cdots \otimes M)\c X _{\theta(1)}(\omega)=\c X_{\theta(1)}(\omega)(M\otimes \cdots \otimes M)$. 

Similarly one can check
 that for every $(x_1,\ldots, x_m)\omega \in H_\theta\wr {\sf S}_m$
 $$(\tilde{ \mu'}_{ \gamma \sigma})((x_1, \ldots, x_m)\omega)=\prod_{i=1}^m({\mu'}_{y\sigma})(x_i).$$ 
 Since ${\mu'}_{y\sigma}$ is the restriction of 
 $\mu_{y\sigma}$ the proof is finished. 
\end{proof}

 The following is a special feature of $\c H$-triples with respect to wreath products.
 \begin{thm}\label{wreath2}
 Let $(G, N, \theta)_{\c H}$ and $(H, M, \varphi)_{\c H}$ be $\c H$-triples such that 
 $(G, N, \theta)_{\c H}\geq_c (H, M, \varphi)_{\c H}$. 
 Let $k,m \in \ZZ_{>0}$. Let $\sigma_i \in \c H$ for $i=1, \ldots, k$. 
 Write $\theta_i=\theta^{\sigma_i}$ and $\varphi_i=\varphi^{\sigma_i}$. 
 Suppose that $\theta_i$ and $\theta_j$ are not $G$-conjugate 
 whenever $i\neq j$. Then for $n=mk$
 $$((G\wr {\sf S}_n)_{\tilde \theta^{\c H}}, N^n, \tilde \theta)_{\c H}\geq_c 
 ((H\wr {\sf S}_n)_{\tilde \varphi^{\c H}}, M^n, \tilde\varphi)_{\c H},$$ 
 where $\tilde \theta=\theta_1^m \times \cdots \times \theta_k^m$ and 
 $\tilde \varphi =\varphi_1^m\times \cdots \times \varphi_k^m$.
 \end{thm}
 \begin{proof}
 The statement makes sense since $G\wr {\sf S}_n=N^n(H\wr {\sf S}_n)$, 
 $N^n\cap (H\wr {\sf S}_n)=M^n$ and 
 $\cent {G\wr {\sf S}_n}{N^n}\sbs \cent G N ^n \sbs H^n \sbs H\wr {\sf S}_n$.  
 Moreover, we see next 
 that $(H\wr {\sf S}_n\times \c H)_{\tilde \theta}=(H\wr {\sf S}_n\times \c H)_{\tilde \varphi}$. 
 Write 
 $\tilde \theta=\beta_1\times \cdots \times\beta_n$ and 
 $\tilde \varphi=\xi_1\times \cdots \times \xi_n$.
  We know that each $\beta_i$ is $\theta^{\tau_i}$ for some 
  $\tau_i \in \c H$ and then $\xi_i=\varphi^{\tau_i}$. 
  Let $a\in (H\wr {\sf S}_n\times \c H)_{\tilde \theta}$. 
  Hence $a=(\gamma, \tau)$, where 
  $\gamma=(a_1, \ldots, a_n)\omega\in H\wr {\sf S}_n$ and 
  $\tau \in \c H$. 
  The equality $\tilde \theta^ a =\tilde \theta$ implies that 
  $\beta_{\omega^{-1}(i)}^{a_{\omega^{-1}(i)}\tau}=\beta_i$ 
  for every $i=1, \ldots, n$. This is exactly the same as
 $$\theta^{a_j \tau_j \tau \tau_{\omega(j)}^{-1}}=\theta,$$
 for every $j=1, \ldots, n$. 
 Write $c_j=a_j \tau_j \tau \tau_{\omega(j)}^{-1}\in (H\times \c H)_\theta=(H\times \c H)_\varphi$. 
 Then $\varphi^{c_j}=\varphi$ for every $j=1, \ldots, n$ implies $\tilde \varphi^ a=\tilde \varphi$. 
 The above discussion shows that conditions (i) and (ii) of Definition \ref{Hiso} are satisfied 
 by the $\c H$-triples $((G\wr {\sf S}_n)_{\tilde \theta^{\c H}}, N^n, \tilde \theta)_{\c H}$
 and  $((H\wr {\sf S}_n)_{\tilde \varphi^{\c H}}, M^n, \tilde\varphi)_{\c H}$.

 \smallskip
 
 Next we explain how to construct projective representations 
 giving the relation between the afore-mentioned $\c H$-triples. 
 Let $(\c P, \c P')$ be associated with $(G, N, \theta)_{\c H}\geq_c (H, M, \varphi)_{\c H}$. 
 As in Lemma \ref{wreath} we can construct projective representations $\tilde{\c P}$ and $\tilde{\c P '}$ associated with 
 \begin{equation}\label{eq_wreath}
 ( (G_\theta \wr {\sf S}_m)\Delta^m G, N^m, \psi)_{\c H}\geq_c ((H_\theta \wr {\sf S}_m)\Delta^m H, M^m, \xi)_{\c H},
 \end{equation}
 where $\psi=\theta^m$ and $\xi=\varphi^m$. Write $\tilde{\c P_i}=(\tilde{\c P})^{\sigma_i}$ 
 and $\tilde{\c P'_i}=(\tilde{\c P'})^{\sigma_i}$ for $i=1, \ldots, k$.
 In particular, each $(\tilde{\c P_i}, \tilde{\c P'_i})$ gives
 $$( G_\theta \wr {\sf S}_m, N^m, \psi_i)\geq_c (H_\theta \wr {\sf S}_m, M^m, \xi_i),$$
 where $\psi_i=\psi^{\sigma_i}=\theta_i^m$ and $\xi_i=\xi^{\sigma_i}=\varphi_i^m$.
 The pair $(\tilde{\c P}, \tilde{\c P'})$ of tensor product representations
 $$\tilde{\c P}=\tilde {\c P_1}\otimes \cdots \otimes \tilde{\c P_k} \text{ \ and \ } \tilde{\c P'}=\tilde {\c P'_1}\otimes \cdots \otimes \tilde{\c P'_k}$$
 gives
 $$( (G_\theta \wr {\sf S}_m)^k, N^n, \tilde \theta)\geq_c ((H_\theta \wr {\sf S}_m)^k, M^n, \tilde \varphi).$$
 Notice that $(G_\theta \wr {\sf S}_m)^k=(G\wr {\sf S}_n)_{\tilde \theta}$ and 
 $(H_\theta \wr {\sf S}_m)^k=(H\wr {\sf S}_n)_{\tilde \varphi}$. 
 This is because $\theta_i$ and $\theta_j$ are not $G$-conjugate whenever $i\neq j$. 
 Notice that we have constructed $\tilde{ \c P}$ and $\tilde {\c P'}$ as in
 Definition \ref{Hiso}(iii). 
 
 \smallskip
 
It only remains to check condition (iv) of Definition \ref{Hiso}. 
As before write $\tilde \theta=\beta_1\times \cdots \times\beta_n$. Note that
 $\beta_i=\theta_j$ whenever $i \in  \Lambda_j=\{(j-1)m+1, \ldots, jm\}$, for $j=1, \ldots, k$. 

\smallskip

Let $a\in (H\wr {\sf S}_n\times \c H)_{\tilde \theta}$. 
Hence $a=(\gamma, \tau)$, where $\gamma=(a_1, \ldots, a_n)\omega\in (H\wr {\sf S}_n)_{\tilde \theta^{\c H}}$ and 
$\tau \in \c H$. The equality $\tilde \theta^ a=\tilde \theta$ is equivalent to
$$\beta_{\omega^{-1}(i)}^{a_{\omega^{-1}(i)}\tau}=\beta_i$$
for every $i=1, \ldots, n$. In particular
$$\beta_{\omega^{-1}(i)}^{a_{\omega^{-1}(i)}\tau}=\theta_j \text{ \ \ whenever \ \ } i \in \Lambda_j.$$
Hence $\omega^{-1}(\Lambda_j)=\Lambda_l$ for some $l \in \{ 1, \ldots, k\}$, and 
$\theta^{\sigma_l a_i \tau}=\theta_i^{a_i\tau}=\theta_j=\theta^{\sigma_j}$ 
for every $i \in \Lambda_l$. 
In particular $\sigma_l\tau\sigma_j^{-1}\in \c H _{H, \theta}$. 
Fix for each  $l \in \{ 1, \ldots, k\}$ an element $c_l\in H$ 
such that $\theta^{c_l}=\theta^{\sigma_l\tau\sigma_j^{-1}}$. 
Hence $a_i=a_i'c_l$ for some $a_i'\in H_\theta$ for every 
$i \in \Lambda_l$ and for $l=1, \ldots, k$. Write $b_l=\Delta^mc_l$ for each $l$.

Define $\pi \in {\sf S}_n$ by $\pi((l-1)m+i)=(j-1)m+i$ if 
$\omega(\Lambda_l)=\Lambda_j$ for every $i=1, \ldots, m$. 
Hence $\pi(l)=j$ if $\omega(\Lambda_l)=\Lambda_j$, and in this way
we can view $\pi \in {\sf S}_k$. 
For $j =1, \ldots, k$,
define $\pi_j \in {\sf S}_n$ by $\pi_j|_{\Lambda_j}=\omega\pi^{-1}|_{\Lambda_j}$ 
and fixing $\{ 1, \ldots, n\}\setminus \Lambda_j$. By definition $\omega=\pi_1\cdots\pi_k \pi$. 

\smallskip 

Then $\gamma=x \gamma'$, where 
$x=(a'_1, \ldots, a'_n)\pi_1\cdots \pi_k \in (H_\theta \wr {\sf S}_m)^k$ and 
$\gamma' =(b_1, \ldots, b_k)\pi \in (H\wr {\sf S}_n)_{\tilde \theta^{\c H}}$ 
(this is because $(b_1, \ldots, b_k)$ and $\pi_1\cdots \pi_k$ commute). 
By Lemma \ref{cor_trans}, in order to verify condition (iv) of Definition 
\ref{Hiso} for $\gamma$ we may assume that $\gamma=\gamma'$.

With the above assumptions and notation, 
we have that $\psi^{b_l\sigma_l\tau\sigma_{\pi(l)}^{-1}}=\psi$ for all $l$. 
Note that $\gamma^{-1}=(b_{\pi^{-1}(1)}^{-1}, \ldots, b_{\pi^{-1}(k)}^{-1})\pi^{-1}$. 
Let $(y_1, \ldots, y_k)\in (H_\theta\wr{\sf S}_m)^k$. Then we have
\begin{align*}
(\tilde{\c P_1}\otimes \cdots \otimes\tilde{\c P_k})^{\gamma\tau}(y_1, \ldots, y_k)&=\tilde{\c P}^{\sigma_1\tau}\otimes \cdots \otimes\tilde{\c P}^{\sigma_k \tau}(y_{\pi(1)}^{ b_1^{-1}}, \ldots, y_{\pi(k)}^{ b_k^{-1}})\\[1.5mm]
 &=\tilde{\c P}^{\sigma_1\tau b_1}(y_{\pi(1)})\otimes \cdots \otimes  \tilde{\c P}^{\sigma_k\tau b_k}(y_{\pi(k)})\\[1.5mm]
 &=(\tilde{\c P}^{b _1\sigma_1\tau\sigma_{\pi(1)}^{-1}}(y_{\pi(1)}))^{\sigma_{\pi(1)}}\otimes \cdots \otimes  (\tilde{\c P}^{b_k\sigma_k\tau\sigma_{\pi(k)}^{-1}}(y_{\pi(k)}))^{\sigma_{\pi(k)}}.
\end{align*}
Write $\tau_l=\sigma_l \tau \sigma_{\pi(l)}^{-1}$. For each $l\in \{1, \ldots, k\}$, 
we have that $\psi^{b_l \tau_l}=\psi$. 
By Lemma \ref{mu}, we have functions $\tilde \mu_{b_l \tau_l }$ 
such that 
\begin{equation}\label{similarity}(\tilde{\c P})^{ b_l\tau_l}\sim \tilde \mu_{ b _l\tau_l}\tilde{\c P}.\end{equation}
Write
$$\tilde \mu_{\gamma \tau} (y_1, \ldots, y_k):=\prod_{l=1}^k \tilde \mu _{ b _l \tau _l}(y_{\pi(l)})^{\sigma_{\pi(l)}}.$$
Then
\begin{align*}
(\tilde{\c P}^{\sigma_1}\otimes \cdots \otimes\tilde{\c P}^{\sigma_k})^{\gamma\tau}(y_1, \ldots, y_k)&\sim \tilde \mu_{\gamma \tau} (y_1, \ldots, y_k) \tilde{\c P}^{\sigma_{\pi(1)}}(y_{\pi(1)})\otimes \cdots \otimes  \tilde{\c P}^{\sigma_{\pi(k)}}(y_{\pi(k)})\\[1.5mm]
&\sim \tilde \mu _{\gamma\tau}(y_1, \ldots, y_k) \tilde{\c P}^{\sigma_1}\otimes \cdots \otimes  \tilde{\c P}^{\sigma_k}(y_1, \ldots, y_k)\, ,
\end{align*}
where the first similarity relation follows from the ones in Equation (\ref{similarity}), and the second similarity
relation is obtained by conjugating by the matrix
$\c X (\pi)$ associated with the action of $\pi$ on the tensors
$$v_1\otimes \cdots \otimes v_k \mapsto v_{\pi^{-1}(1)}\otimes \cdots \otimes v_{\pi^{-1}(k)}.$$
We have analogous relations for $\tilde{\c P'_1}\otimes \cdots \otimes\tilde{\c P'_k}$ with 
$$\tilde{ \mu'}_{\gamma \tau} (y_1, \ldots, y_k):=\prod_{l=1}^k \tilde {\mu'} _{ b _l \tau _l}(y_{\pi(l)})^{\sigma_{\pi(l)}}.$$
 By Equation (\ref{eq_wreath}) (in the second paragraph of this proof) each 
 $\tilde{ \mu'} _{\tilde b _j \tau _j}$ is the restriction of 
 $\tilde \mu _{\tilde b _j \tau _j}$, and hence the result follows.
 \end{proof}
 
 We will need to control the character theory and $\c H$-action over 
 some characters of central products. 
Suppose that $K$ is the product of two subgroups $N$ and $Z$ with $N\nor K$ and $Z \leq \cent K N$.
Then $K$ is the central product of $N$ and $Z$. 
In this case 
$$\irr K =\dot{ \bigcup_{\nu \in \irr{Z\cap N}}} \irr{K|\nu},$$
where $\irr{K|\nu}=\{ \theta\cdot \lambda \ | \ \theta \in \irr{N|\nu} \text{ and } \lambda \in \irr{Z|\nu} \}$. 
Note that, whenever a group $A$ acts by automorphisms on $K$,
  stabilizing $N$ and $Z$, if $a \in A$ and 
$\theta\cdot \lambda \in \irr K$, then 
$(\theta \cdot \lambda)^a=\theta\cdot \lambda$ if, and only if,
$\theta^ a=\theta$ and $\lambda^a=\lambda$. The same happens if $A \le \c G={\rm Gal}(\Q^{\rm ab}/\Q)$.

\begin{thm}\label{dotproduct}  Let $(G, N, \theta)_{\c H}$ and 
$(H, M, \varphi)_{\c H}$ be $\c H$-triples such that 
$$(G, N, \theta)_{\c H}\geq_c (H, M, \varphi)_{\c H}\, .$$
Suppose that $Z\nor G$ is abelian and satisfies  $Z\sbs \cent G N$.
Let $\nu \in \irr{Z\cap N}$ be under $\theta$ and $\lambda \in \irr {Z|\nu}$. Then
$$(G_{(\theta\cdot \lambda)^\c H}, NZ, \theta\cdot \lambda)_{\c H}\geq_c (H_{(\varphi\cdot \lambda)^\c H}, MZ, \varphi\cdot \lambda)_{\c H}.$$
In particular,   
there exists an $\c H$-equivariant bijection
$$\tau_\theta \colon \irr{G_{(\theta\cdot \lambda)^{\c H}}| (\theta\cdot\lambda)^{\c H}}\rightarrow \irr{H_{(\varphi\cdot \lambda)^{\c H}}|(\varphi\cdot\lambda}^{\c H})$$
that preserves ratios of character degrees. 
\end{thm}
\begin{proof} Note that  $Z\sbs \cent G N \sbs H$. Hence $Z\cap N=Z \cap M$. 
Since $(G, N, \theta)_{\c H}\geq_c (H, M, \varphi)_{\c H}$, we have that $\varphi$ lies over $\nu$.
 
Note that $(\theta \cdot \lambda)^{g\sigma}=\theta \cdot \lambda$
if, and only if, $\theta^{g\sigma}=\theta$ and $\lambda^{g\sigma}=\lambda$, for $g\in G$ and $\sigma \in \c H$.
Hence, $G_{(\theta\cdot \lambda)^\c H} \cap H=H_{(\varphi\cdot \lambda)^\c H}$,
and Definition \ref{Hiso}(i) is easily checked. Since $(H\times \c H)_\theta=(H\times \c H)_\varphi$,
we have that $(H_{(\varphi\cdot \lambda)^\c H}\times \c H)_\theta=(H_{(\varphi\cdot \lambda)^\c H}\times \c H)_\varphi$, so 
Definition \ref{Hiso}(ii) holds.

Since $G_{(\theta\cdot \lambda)}=
G_\theta \cap G_\lambda \sbs G_\theta$
and $H_{\theta\cdot \lambda}=G_{\theta\cdot \lambda}=H_{\varphi \cdot \lambda}$, 
if $(\c P, \c P')$ gives
$$(G, N, \theta)_{\c H}\geq_c (H, M, \varphi)_{\c H},$$
then $(\c P_{G_{\theta \cdot \lambda}}, \c P'_{H_{\theta \cdot \lambda}})$ gives
$$(G_{(\theta \cdot \lambda)^{\c H}}, N, \theta)_{\c H}\geq_c (H_{(\theta \cdot \lambda)^{\c H}}, M, \varphi)_{\c H}.$$
To ease the notation we assume $G=G_{(\theta \cdot \lambda)^{\c H}}$, so that $(G, NZ, \theta \cdot \lambda)_{\c H}$ and $(H, MZ, \varphi \cdot \lambda)_{\c H}$ are $\c H$-triples.

\smallskip

We next show how to construct a pair of projective representations giving 
$$(G, NZ, \theta\cdot \lambda)_{\c H}\geq_c (H, MZ, \varphi\cdot \lambda)_{\c H}$$
from $(\c P, \c P')$. Notice that $\c P(c)=\zeta_c I_{\theta(1)}$ and $\c P'(c)=\zeta_c I_{\varphi(1)}$ for every $c \in \cent G N$. 
Morever, $\zeta_z=\nu(z)$ for every $z \in Z \cap N$.

By 
Corollary \ref{projectivecyclotomic}, we have that $\zeta_c \in U$ for every $c \in \cent G N$, with $U$ being 
 the subgroup of $\C^\times$ of roots of unity.
Let $\epsilon \colon G_\theta \rightarrow U$ be a function constant on $N$-cosets and such that
$$\epsilon(z)=\frac{\lambda(z)}{\zeta_z}\in U \, ,$$
for every $z \in Z$. Such $\epsilon$ does exist as $\lambda(z)=\nu(z)=\zeta_z$ for every $z \in Z \cap N$. In particular, $\epsilon (1)=1$. 
By Lemma \ref{replace},  the pair $(\epsilon \c P, \epsilon_{H_\theta} \c P')$ gives 
$(G, N, \theta)_{\c H}\geq_c (H, M, \varphi)_{\c H}.$
Note that $(\epsilon \c P)_{NZ}$ affords $\theta \cdot \lambda$ and $(\epsilon_{H_\theta} \c P')_{MZ}$ affords $\varphi \cdot \lambda$. 
In particular,  $(\epsilon_{G_{\theta \cdot \lambda}} \c P, \epsilon_{H_{\theta \cdot \lambda}} \c P')$ gives 
$$(G, NZ, \theta\cdot \lambda)_{\c H}\geq_c (H, MZ, \varphi\cdot \lambda)_{\c H}\, .$$
To finish the proof apply Theorem \ref{Hequivariant}. 
\end{proof}

The following is an $\c H$-triple version of the most important result 
concerning the application of the theory 
of centrally isomorphic character triples to reduction theorems:
the ordering of two $\c H$-triples {\em only} depends on the automorphisms of 
the normal subgroup defined via conjugation by the overgroup,
see the {\em butterfly} theorem (Theorem 5.3 in  \cite{S}).

\begin{thm}\label{butterfly}
Let $(G, N, \theta)_{\c H}$ and $(H, M, \varphi)_{\c H}$ be 
$\c H$-triples such that $(G, N, \theta)_{\c H}\geq_c (H, M, \varphi)_{\c H}$. 
Let $\epsilon \colon G \rightarrow \aut N$ be the group homomorphism 
defined by conjugation by $G$. Suppose that $N\nor \hat G$ and 
$\hat \epsilon(\hat G)=\epsilon(G)$, where $\hat \epsilon \colon \hat G \rightarrow \aut N$ 
is given by conjugation by $\hat G$. Let $N \cent{\hat G} N\sbs \hat H\leq \hat G$ 
be such that $\hat \epsilon(\hat H)=\epsilon(H)$. Then 
$$(\hat G, N, \theta)_{\c H}\geq_c (\hat H, M, \varphi)_{\c H}.$$
\end{thm} 
\begin{proof} 
Note that $\cent G N\sbs H$ and $\cent{\hat G} N\sbs \hat H$, 
so the group theory conditions in Definition \ref{Hiso}(i) are satisfie
 by Theorem 10.18
of \cite{Nav18}. 

Recall that the map $\bar\epsilon\colon
G/\cent GN \rightarrow \hat G/\cent{\hat G} N$
given by $\bar\epsilon(\cent GNx)=\cent{\hat G}Ny$ whenever $\epsilon(x)=\hat \epsilon(y)$
defines a group isomorphism. 
Let $x \in G$ and $y \in \hat G$.
If $\epsilon(x)=\hat \epsilon(y)$, notice that 
$\theta^x=\theta^\sigma$
for some  $\sigma \in \c H$
if, and only if, $\theta^{y}=\theta^\sigma$, so condition (ii) of 
Definition \ref{Hiso} also holds.

Following Theorem 10.18 of \cite{Nav18} and given a tranversal 
$\mathbb T$ of $M\cent G N$ in $H_\varphi$ with $1 \in \mathbb T$, 
we can define a transversal $\widehat {\mathbb T}$ of 
$M\cent {\hat G} N$ in $\hat H_\varphi$ with $\hat 1 = 1$. 

Suppose that $(\c P, \c P')$ gives $(G, N, \theta)_{\c H}\geq_c(H, M, \varphi)_{\c H}$
and 
let $\lambda \colon \cent G N \rightarrow \Q^{\rm ab}\setminus\{0\}$ 
be given by the scalar associated with $\c P$ and $\c P'$ for every $c \in \cent G N$ . 
Choose a function $\hat \lambda \colon \cent {\hat G} N \rightarrow \Q^{\rm ab}\setminus\{0\}$ 
semi-multiplicative with respect to $\zent N$ as in Theorem 10.18 of \cite{Nav18}. 
Also following Theorem 10.18 of \cite{Nav18},
we can construct projective representations $\hat{\c P}$ of $\hat G_{\theta}$ 
and $\hat{\c P'}$ of $\hat H_\varphi$ with respect to $\widehat{\mathbb T}$, 
$\hat \lambda$ and $\c P$ and $\c P'$ respectively. 
Then $\hat{\c P}$ is associated with $\theta$, $\hat{\c P'}$ is associated with $\varphi$, 
and they satisfy Definition \ref{Hiso}(iii). 

It remains to check Definition \ref{Hiso}(iv) for $(\hat{\c P}, \hat{\c P'})$. 
We have that $H$ acts {\em by conjugation} 
on the transversal $\mathbb T$ with $t \mapsto t\cdot h$ if, and only if, 
$(t\cdot h)^{-1}t^h=m_h c_h \in M\cent G N$ for $h \in H$. Similarly $\hat H$ acts on 
$\widehat{\mathbb T}$. In fact, from the definition of $\widehat{\mathbb T}$ 
it follows that if $\epsilon(h)=\hat \epsilon(\hat h)$ for $h \in H$, $\hat h \in \hat H$, then 
$$\hat t \cdot \hat h=\widehat{t \cdot h}$$ 
and $(\widehat{t \cdot h})^{-1}\hat {t}^{\hat h}=m_h \hat c_{\hat h}\in M\cent{G_1} N$, for every $t \in \mathbb T$. 

Let $t \in \mathbb T$, $m \in M$ and $\hat c\in \cent {\hat G} N$. Then
\begin{align*}
\hat{\c P}^{\hat h^{-1}\sigma}(\hat t m \hat c)&= \hat{\c P}(\hat t^{\hat h}m^{h} (\hat c)^{\hat h})^\sigma\\ 
&= \hat{\c P} (\widehat{t\cdot h}m_{h} \hat c_{\hat h}m^{h}(\hat c)^{\hat h})^\sigma\\
& = \c P (t\cdot h)^\sigma\c P (m_hm^h)^\sigma \hat \lambda(\hat c_{\hat h} (\hat c)^{\hat h})^\sigma\\
&=\c P (t^hm_h^{-1}c_h^{-1}m_hm^h)^\sigma  \lambda(\hat c_{\hat h} (\hat c)^{\hat h})^\sigma\\
& = \c P ((tm)^h)^\sigma\lambda(c_h^{-1})^\sigma  \hat \lambda(\hat c_{\hat h} (\hat c)^{\hat h})^\sigma \\
&= \c P^{h^{-1}\sigma} (tm)\lambda(c_h^{-1})^\sigma   \hat \lambda(\hat c_{\hat h} (\hat c)^{\hat h})^\sigma \\
&\sim \mu_{h^{-1}\sigma}(tm)  \lambda(c_h^{-1})^\sigma    \hat \lambda(\hat c_{\hat h} (\hat c)^{\hat h})^\sigma\c P(tm)\hat \lambda(\hat c)\hat \lambda(\hat c)^{-1}\\
&= \mu_{h^{-1}\sigma}(tm) \lambda(c_h^{-1})^\sigma    \hat \lambda(\hat c_{\hat h} (\hat c)^{\hat h})^\sigma\lambda(\hat c)^{-1}\hat{\c P} (\widehat tmc')  \\
& = \hat \mu_{\hat h^{-1}\sigma}(\hat t m \hat c)\hat{\c P} (\hat t m \hat c)\,
\end{align*}
where $\hat \mu_{\hat h^{-1}\sigma}(\hat t m \hat c)= \mu_{h^{-1}\sigma}(tm) \lambda(c_h^{-1})^\sigma    \hat \lambda(\hat c_{\hat h} (\hat c)^{\hat h})^\sigma\lambda(\hat c)^{-1}$.
Similarly
 \begin{align*}
 (\hat{\c P'})^{\hat h^{-1}\sigma}(\hat t m \hat c)&\sim \hat \mu'_{\hat h^{-1}\sigma}(\hat t m \hat c)\hat{\c P'} (\hat t m \hat c),
 \end{align*} 
where $\hat \mu'_{\hat h^{-1}\sigma}(\hat t m \hat c)=\mu'_{h^{-1}\sigma}(tm) \lambda(c_h^{-1})^\sigma    \hat \lambda(\hat c_{\hat h} (\hat c)^{\hat h})^\sigma\lambda(\hat c)^{-1}$, so that $\hat \mu_{\hat h^{-1}\sigma}$ and $\hat \mu'_{\hat h ^{-1}\sigma}$ agree on $\hat H_\varphi$
 provided that $\mu_{h^{-1}\sigma}$ and $\mu'_{h^{-1}\sigma}$ agree on $H_\varphi$.
\end{proof}

\section{The inductive Galois--McKay condition}\label{inductivecondition}

We refer the reader to the Appendix B of \cite{Nav18}
for a compendium of the definitions
and results on the theory of universal covering groups 
that are specifically needed in our context.

\smallskip

We can now define the {\bf inductive Galois--McKay condition} 
on finite non-abelian simple groups. 

\begin{defi}\label{inductive}Let $S$ be a finite non-abelian simple group, 
with $p$  dividing  $|S|$. Let $X$ be a universal covering group 
of $S$, $R \in \syl p X$ and $\Gamma =\aut X _R$. We say that 
$S$ satisfies the inductive Galois--McKay condition for $p$ 
if there exist some $\Gamma$-stable proper subgroup $N$ of $X$ with 
$\norm X R \sbs N$ and some $\Gamma\times \c H$-equivariant bijection 
$$\Omega \colon \irrp {X}\rightarrow \irrp{N}\, ,$$ such that 
for every $\theta \in \irrp X$ we have
$$(X\rtimes \Gamma_{\theta^{\c H}}, X, \theta)_{\c H}\geq_c(N\rtimes \Gamma_{\theta^{\c H}}, N, \Omega(\theta))_{\c H}.$$
\end{defi}

We recall that for a quasisimple group $X$ 
(a perfect group whose quotient by its center is simple) 
and any $n \in \Z_{>0}$,
$$\aut{X^n}=\aut X \wr {\sf S}_n \, .$$
Moreover, whenever $R\leq X$ 
$$\aut{X^n}_{R^n}=\aut X_R \wr {\sf S}_n \, .$$
These results appear as Lemma 10.24 of \cite{Nav18}, for example.

\begin{thm} \label{consindcond}
Suppose that $S$ satisfies the inductive Galois--McKay condition 
for $p$, and $X$, $R$, $\Gamma$, $N$ and $\Omega$ are as above. 
Then for any $n \in \ZZ_{> 0}$ the map 
$$\tilde \Omega \colon \irrp{X^n}\rightarrow \irrp {N^n}$$ 
given by $\tilde \Omega (\tilde \theta)=\tilde \varphi$ 
where $\tilde \theta=\theta_1\times \cdots \times \theta_n$ 
and $\tilde \varphi=\varphi_1\times \cdots \times \varphi_n$ 
with $\Omega(\theta_i)=\varphi_i$ is a bijection. 
Write $\tilde \Gamma =\Gamma \wr {\sf S}_n$. 
Then $\tilde \Omega $ is $\tilde \Gamma\times \c H$-equivariant 
and for every $\tilde\theta \in \irrp {X^n}$
$$(X^n\rtimes \tilde \Gamma_{\tilde \theta^{\c H}}, X^n, \tilde \theta)_{\c H}\geq_c(N^n\rtimes \tilde \Gamma_{\tilde \theta^{\c H}}, N^n, \tilde \varphi)_{\c H}.$$
\end{thm}

\begin{proof} Notice that $\tilde \Gamma=\aut{X^n}_{R^n}$. The fact that $\tilde \Omega$ is a 
$\tilde \Gamma \times \c H$-equivariant bijection 
follows straightforwardly from definitions. Also notice that $X^n\rtimes \tilde \Gamma=X^n(N^n\rtimes \tilde \Gamma)$ 
and $X^n \cap (N^n \rtimes \tilde \Gamma)=N^n$.

Given $\tilde \theta \in \irrp {X^n}$, we prove the statement
on $\c H$-triples in a series of 
steps concerning assumptions on $\tilde \theta$. 
Note that, by applying Lemma \ref{conjugate} we can replace $\tilde \theta$ by 
any $\tilde \Gamma$-conjugate of $\tilde \theta$.

If $\tilde \theta=\theta_1 \times \cdots \times \theta_n$, then we write $\c H_{\Gamma_i}$
for the subgroup of $\c H$ such that $\c H_{\Gamma_i}/\c H_{\theta_i}$ is the image of the natural monomorphism $\Gamma_{\theta_i^{\c H}}/\Gamma_{\theta_i}\rightarrow \c H /\c H_{\theta_i}$. We refer to the $\theta_i$ as factors of $\tilde \theta$.

\smallskip

{\em Step 1. We may assume that all factors of $\tilde \theta$ are $\c H$-conjugate and any two $\Gamma$-conjugate (in particular $\c H_{\Gamma_i}$-conjugate) factors are equal.}

By Lemma \ref{conjugate} and after conjugating by an element of $\Gamma^n$, we may assume that any two factors of $\tilde \theta$ are either equal or not $\Gamma$-conjugate. In particular, and after maybe conjugation by an element of ${\sf S}_n$, we can write $\tilde \theta=\bigtimes_{j=1}^\ell \tilde \theta_j$, where all the factors of $\tilde \theta_j$ lie in $\theta_j^{\c H}$. Notice that any two factors of $\tilde \theta _j$ are either equal or not $\c H _{\Gamma_j}$-conjugate. Hence $\tilde \Gamma_{\tilde \theta^{\c H} }=\bigtimes_{j=1}^\ell \Gamma_{\tilde \theta_j^{\c H}}$ and $\Gamma_{\tilde \theta}=\bigtimes_{j=1}^\ell \Gamma_{\tilde \theta_j}$. By Lemma \ref{direct} we may assume $\ell=1$, that is, all the factors of $\tilde \theta$ lie in the same $\c H$-orbit. 

\medskip 

{\em Step 2. We may assume that 
$\tilde \theta=\theta_1^{m}\times \cdots \times \theta_k^{m}$ for some $m$, where $\theta_i=\theta^{\sigma_i}$ for some $\sigma_i\in \c H$ and $\c H_\Gamma \sigma_i\neq \c H_\Gamma \sigma_j$ whenever $i\neq j$. Here $\c H_\Gamma=\c H _{\Gamma_{\theta^{\c H}}, \theta}$ as defined above.}

By Step 1 and after conjugating $\tilde \theta$ by an element of ${\sf S}_n$, we can write $\tilde \theta=(\theta_1)^{n_1}\times \cdots \times (\theta_k)^{n_k}$ where $\theta_i=\theta^{\sigma_i}$ for some $\sigma_ i \in \c H$ and two different $\sigma_i$ define different $\c H_\Gamma$-cosets. We work to show that any element $\gamma \in \tilde \Gamma_{\tilde \theta^{\c H}}$ permutes  $\theta_i$ and $\theta_j$ if, and only if, $n_i=n_j$. If we can show that then, after conjugating $\tilde \theta$ by an element of ${\sf S}_n$, we may decompose $\tilde \theta$ as a direct product of characters $\tilde \psi=\psi_1^s \times \cdots \times \psi_r^s$ which do not have any factor in common (pairwise). In particular, $\tilde \Gamma _{\tilde \theta^{\c H}}$ decomposes as a direct sum of $(\Gamma \wr {\sf S}_{sr})_{\tilde \psi^{\c H}}$. By Lemma \ref{direct} the claim of the step would follow. 

\smallskip

Given $\tilde \xi =\xi_1\times \cdots \times \xi_n \in \irr {X^n}$ such that all $\xi_i$ lie in one $\c H$-orbit, we can associate to $\tilde \xi$ the multiset $[\tilde \xi]$ of the $\c H_{\Gamma, \xi_1}$-orbits of the factors, namely
$[\tilde \xi]=[\xi_1^{\c H _{\Gamma, \xi_1}}, \ldots, \xi_n^{\c H_{\Gamma, \xi_1}}]$, where $\c H_{\Gamma, \xi_1}=\c H_{\Gamma_{\xi_1^{\c H}}, \xi_1}$
Note that for every $\gamma=(a_1, \ldots, a_n)\omega \in \tilde \Gamma$ we have that
$[\xi^\gamma]=[(\xi_1^{a_1})^{\c H_{\Gamma, \xi_1}}, \ldots, (\xi_n^{a_n})^{\c H_{\Gamma, \xi_1}}].$

Write $\tilde \theta=\beta_1 \times \cdots \times \beta_n$, so that each $\beta_i=\theta^{\tau_i}$ for some $\tau_i \in \c H$. 
Note that $\c H_{\Gamma}\tau_i =\c H_{\Gamma}\tau_j$ if, and only if, $\beta_i=\beta_j$.
Let $\gamma=(a_1 \ldots, a_n)\omega \in \tilde \Gamma$, with $a_i \in \Gamma$ and $\omega \in {\sf S}_n$. 
Then the multiplicity of $(\beta_i^{a_i})^{\c H _\Gamma}$ in $[\tilde \theta^\gamma]$ is the same as the multiplicity of $\beta_i^{\c H_\Gamma}$ in $[\tilde \theta]$, which is the same as the number of factors equal to $\beta_i$ in $\tilde \theta$. This is because $\beta_j^{a_j}$ lies in the $\c H_\Gamma$-orbit of $\beta_i^{a_i}$ if, and only if, $\c H_\Gamma \tau_i=\c H_\Gamma \tau _j$ (using that $\Gamma$-conjugate factors of $\tilde \theta$ are equal by Step 1). The latter happens if, and only if, $\beta_i=\beta_j$. 

For $(\gamma, \tau)\in (\tilde \Gamma \times \c H)_{\tilde \theta}$, we have $\gamma=(a_1, \ldots, a_n)\omega \in \tilde \Gamma_{\tilde \theta^{\c H}}$ and $[\tilde \theta^{\gamma \tau}]=[\tilde \theta]$. By the above paragraph the multiplicity of $(\beta_i^{a_i \tau})^{\c H _\Gamma}$ in $[\tilde \theta]=[\tilde \theta^{\gamma \tau}]$ equals the multiplicity of $\beta_i^{\c H_\Gamma}$ in $[\tilde \theta]$, that is, the number of factors equal to $\beta_i$ in $\tilde \theta$.
On the other hand $\tilde \theta^{\gamma \tau}=\tilde \theta$ if, and only if,  $\beta_j=\beta_i^{a_i\tau}$ whenever $\omega(i)=j$. Putting these two facts together we see that if $\omega(i)=j$ then,  the number of factors equal to $\beta_j$ in $\tilde \theta$ is the same as the number of factors equal to $\beta_i$ in $\tilde \theta$, as wanted.  

\medskip

{\em Final step.} By Step 2 we have that $\tilde \theta=\theta_1^m\times \cdots \times \theta_k^m$, where $\theta_i=\theta^{\sigma_i}$ for some $\sigma_i\in \c H$ and $\sigma_i$ and $\sigma_j$ define distinct $\c H_\Gamma$-cosets whenever $i\neq j$. The result then follows by applying Theorem \ref{wreath2} with $X \nor G=X\rtimes \Gamma_{\theta}$. Note that the condition that $\sigma_i$ and $\sigma_j$ define distinct $\c H_\Gamma$-cosets whenever $i\neq j$ is equivalent to saying that no $\theta_i$ is $X\rtimes \Gamma_{\theta^{\c H}}$-conjugate to $\theta_j$ if $i\neq j$.
\end{proof}

\begin{thm}\label{perfectcentralextension}
Suppose that $K\nor G$, where $K$ is perfect and $K/\zent K$ is 
isomorphic to a direct product of copies of a non-abelian 
simple group $S$. Let $Q \in \syl p K$. Assume that $S$ satisfies 
the inductive Galois--McKay condition for $p$. Then there exists an 
$\norm G Q$-invariant subgroup $\norm K Q \sbs M <K$ and an 
$\norm G Q \times \c H$-equivariant bijection 
$$\Omega \colon \irrp{K}\rightarrow \irrp{M} \, ,$$
 such that for every $\theta \in \irrp K$, we have 
$$(G_{\theta^{\c H}}, K, \theta)_{\c H}\geq_c (\norm G M_{\theta^{\c H}}, M,\Omega(\theta))_{\c H}\, .$$
\end{thm}
\begin{proof}
First note that the $\c H$-triples relations make sense. By the Frattini argument
$G=K \norm G Q$. Since $\norm G Q \sbs \norm G M$ by the assumptions on $M$
then $G=K\norm G M$. Also by the Frattini argument
$\norm G M=M\norm G Q$. 
Hence $\norm K M=K\cap \norm G M=M\norm K Q=M$. 

Notice that if the theorem is true for $Q$, then it is true for $Q^k$ for any $k \in K$. 
This is because $\Omega_k(\theta):=\Omega(\theta)^k$ would be $\norm G Q^k\times \c H$-equivariant
and by using Lemma \ref{conjugate}. Hence we may choose any Sylow $p$-subgroup of $K$.

Let $X$ be the universal covering of $S$, and let $\pi\colon X^n \rightarrow K$ be a covering of $K$ with $Z=\ker \pi \sbs \zent {X^n}$.
Since $S$ satisfies the inductive Galois--McKay condition for $p$,  we have $R$, $N$ and $\Omega$ given by Definition \ref{inductive}.
We prove the result with respect to $\pi(R^n)=Q\in \syl p K$. Write $M=\pi (N^n)\supseteq \norm K Q$.

The idea is to prove the theorem with respect to $K\nor \hat G=K\rtimes \aut K_Q$ and to use Theorem \ref{butterfly} to relate 
$G$ and $\hat G$ via their conjugation homomorphisms into $\aut K$.

We mimic the proof of Theorem 10.25 in \cite{Nav18}, see there for more details. 
Write $\Gamma=\aut X _R$ as in Definition \ref{inductive}
and $\tilde \Gamma=\Gamma \wr {\sf S}_n$.
 By Theorem \ref{consindcond} we have a $\tilde \Gamma\times \c H$-equivariant
bijection
$$\tilde \Omega \colon \irrp{X^n}\rightarrow \irrp {N^n}$$ 
such that for every $\tilde\theta \in \irrp {X^n}$
$$(X^n\rtimes \tilde \Gamma_{\tilde \theta^{\c H}}, X^n, \tilde \theta)_{\c H}\geq_c(N^n\rtimes \tilde \Gamma_{\tilde \theta^{\c H}}, N^n, \tilde \varphi)_{\c H}.$$
Since these are, in particular, central isomorphisms and $Z\sbs \zent{X^n}$ then $\tilde \theta$ lies over $1_Z$ if, and only if, $\tilde \varphi$ lies over $1_Z$. 
Let $B=\tilde \Gamma_Z \leq  \tilde\Gamma$. In particular, $\tilde \Omega$ yields a bijection that we denote by $\tilde \Omega$ again
$$\tilde \Omega \colon \irrp{X^n/Z}\rightarrow \irrp {N^n/Z}$$ 
which is $B\times \c H$-equivariant. By Lemma \ref{quotient}
$$(X^n/Z\rtimes  B_{\tilde \theta^{\c H}}, X^n, \tilde \theta)_{\c H}\geq_c(N^n/Z\rtimes B_{\tilde \theta^{\c H}}, N^n, \tilde \varphi)_{\c H},$$
for every $\tilde \theta \in \irrp{X^n/Z}$ and $\tilde \varphi=\tilde \Omega(\theta)$.

Note that $X^n/Z \times B\cong  \hat G$ via $\pi$ and under this isomorphism $N^n/Z \rtimes B$ corresponds to $M\rtimes \aut K _Q$. Hence we have proven that there exists an $\aut K_Q \times \c H$-equivariant bijection
$$\Omega \colon \irrp K \rightarrow \irrp M.$$
Moreover, by Lemma \ref{iso}
$$(K\rtimes (\aut K_Q)_{\tilde \theta^{\c H}}, K, \theta)_{\c H}\geq_c(M\rtimes (\aut K_Q)_{\tilde \theta^{\c H}}, M, \Omega(\theta))_{\c H},$$
whenever $\theta\in \irrp K$.

To finish the proof  apply Theorem \ref{butterfly}
as in the end of the proof of Theorem 10.25 of \cite{Nav18}:
Let $\epsilon \colon G \rightarrow \aut K$ and $\hat \epsilon \colon \hat G \rightarrow \aut K$ be the corresponding conjugation homomorphisms.
Let $\theta \in \irrp K$ and let $V=\epsilon (G_{\theta ^{\c H}})$. The same arguments as in the proof of Theorem 10.25 of \cite{Nav18}
show that if $\hat V :=\hat \epsilon ^{-1}(V)$,  then $\epsilon^{-1}( \epsilon (\norm{\hat V} M))=\norm {G_{\theta^{\c H}}}M=\norm G M_{\theta^{\c H}}$.
\end{proof}

The above result will be key in the reduction 
theorem carried out in Section \ref{thereduction}. 
Below we write the exact form in which it will be later on applied,
in which $K$ (in Theorem \ref{perfectcentralextension}) needs no longer
to be perfect.

\begin{cor}\label{centralextension}
Suppose that $K\nor G$, where $K/\zent K$ is 
isomorphic to a direct product of copies of a non-abelian 
simple group $S$, and let $Q \in \syl p K$.
Assume that $S$ satisfies 
the inductive Galois--McKay condition for $p$. 
If $(G, \zent K, \nu)_{\c H}$ is an $\c H$-triple, then there exists an 
$\norm G Q$-invariant subgroup $\norm K Q \sbs M <K$ and an 
$\norm G Q \times \c H$-equivariant bijection 
$$\Omega \colon \irrp{K|\nu^{\c H}}\rightarrow \irrp{M| \nu^{\c H}} \, ,$$
and for every $\theta \in \irrp {K|\nu^{\c H}}$
$$(G_{\theta^{\c H}}, K, \theta)_{\c H}\geq_c (H_{\varphi^{\c H}}, M, \varphi)_{\c H}\, ,$$
where $\varphi=\Omega(\theta)$ and $H=M\norm G Q$. In particular,
there are character-degree-ratio preserving $\c H$-equivariant bijections 
$$\hat \tau_\theta \colon \irr{G|\theta^{\c H}}\rightarrow \irr{H| \varphi^{\c H}} \, ,$$
satisfying
$\hat \tau_{\theta^h}=\hat \tau_\theta$ and
$\hat \tau_{\theta^\sigma}=\hat \tau_\theta$ for every $h \in H$ and $\sigma \in \c H$.
\end{cor}

\begin{proof} Write $K_1=K'$ and $Z=\zent K$. Hence $K=K_1 Z$ 
is the central product of $K_1$ and $Z$, and $K_1$ is perfect. 
Also $Q_1:=Q\cap K_1\in \syl p {K_1}$. Let $M_1$ 
and $\Omega_1$ be given by Theorem 
\ref{perfectcentralextension} applied with respect to $K_1\nor G$. Note that $Z\cap K_1=Z\cap M_1$.  
Let $\nu_1=\nu_{Z\cap K_1}\in \irr{Z\cap K_1}$, $\theta_1 \in \irrp{K_1|\nu_1^{\sigma}}$ and $\varphi_1=\Omega_1(\theta_1)$.
By Theorem \ref{perfectcentralextension}
$$(G_{\theta_1^{\c H}}, K_1, \theta_1)_{\c H}\geq_c (\norm G {M_1}_{\theta_1^{\c H}}, M_1,\varphi_1)_{\c H}.$$
In particular, this implies that $\varphi _1$ lies over $\nu_1^{\sigma}$ and thus
 $\Omega_1$ maps $ \irrp{K_1|\nu_1^{\c H}}$ onto $ \irrp{M_1|\nu_1^{\c H}}$.
Write $M=M_1Z$, which the central product of $M_1$ and $Z$. Clearly $M$ is $\norm G Q$-invariant as $\norm G{Q}\sbs \norm G {Q_1}$.
 The desired bijection $\Omega$ can be obtained via $\Omega_1$ 
using the dot product of characters as follows (we refer the reader to the discussion preceding Theorem \ref{dotproduct} for more details).  Let $\theta \in \irrp{K|\nu^{\c H}}$ lie 
over $\nu^\sigma$. Then $\theta=\theta_1\cdot \nu^\sigma$ for some
$\theta_1\in \irrp{K_1|\nu_1^{\sigma}}$.  Define $\Omega(\theta):=\Omega_1(\theta_1)\cdot\nu^\sigma=\varphi_1\cdot \nu^\sigma \in \irr{M|\nu^\sigma}$. 
Hence  $$\Omega \colon \irrp{K|\nu^{\c H}}\rightarrow \irrp{M| \nu^{\c H}} \, $$
is an $\norm G Q \times \c H$-equivariant bijection with the desired properties.
Write  $\varphi=\Omega(\theta)$ and recall
$$(G_{\theta_1^{\c H}}, K_1, \theta_1)_{\c H}\geq_c (\norm G {M_1}_{\theta_1^{\c H}}, M_1, \varphi_1)_{\c H}\, .$$
Note that $\norm G{M_1}=H=M\norm G {Q}$, by the Frattini argument, as $Q\sbs M \nor \norm G {M_1}=M_1\norm G{Q_1}$. By Theorem \ref{dotproduct} 
$$(G_{\theta^{\c H}}, K, \theta)_{\c H}\geq_c (H_{\varphi^{\c H}}, M, \varphi)_{\c H}\, .$$
For each $\theta$, denote by $\hat \tau_\theta$ the bijection
provided by Corollary \ref{corHequivariant}. In particular,
$$\hat \tau_\theta \colon \irr {G|\theta^{\c H}}\rightarrow \irr{H| \varphi^{\c H}} $$ is $\c H$-equivariant and preserves ratios of character degrees.
The claims on  $\hat \tau_{\theta^h}$ and $\hat \tau_{\theta^\sigma}$ in the final part of the statement follow from the comments after Lemma \ref{conjugate}  and Lemma \ref{Hconjugate}.
\end{proof}
    
\section{The reduction}\label{thereduction}
 The following key result is due to F. Ladisch. It is based
 on work by  A. Turull. This is an $\c H$-triple version of the 
 well-known fact that a character triple $(G, N, \theta)$ can be replaced by 
 an isomorphic one $(G_1, N_1, \theta_1)$ with $N_1\sbs \zent{G_1}$, in such 
 a way
 that the character properties of $G$ over $\theta$
 are {\em the same} as the character properties
 of $G_1$ over $\theta_1$. If we wish to control
 fields of values of characters above $\theta$, this is no longer true. Still we can 
 somehow replace the original $\c H$-triple by 
 another one with convenient properties. 
  
  \begin{thm}[Ladisch]\label{ladisch}
  Suppose that $(G,Z,\lambda)_{\c H}$ is an $\c H$-triple. Then there exists
  another $\c H$-triple $(G_1, Z_1, \lambda_1)_{\c H}$ such that:
  \begin{enumerate}[{\rm (a)}]
  \item There is an onto homomorphism $\kappa_1\colon G_1 \rightarrow G/Z$
  with kernel $Z_1$.
    \item
  For every $Z \sbs X \le G$, there
  is an $\c H$-equivariant bijection
  $\psi \mapsto \psi_1$ from
  $\irr{X|\lambda^{\c H}} \rightarrow \irr{X_1|\lambda_1^{\c H}}$, where $\kappa_1(X_1)/Z=X/Z$,
  preserving ratios of character degrees; more precisely, if $\psi$ and $\psi_1$ 
  correspond under the above bijection then $\psi(1)/\lambda(1)=\psi_1(1)/\lambda_1(1)$.
  Furthermore, if $g_1 \in G_1$, $g=\kappa(g_1)$ and $\psi \in  \irr{X|\lambda^{\c H}}$,  then
  $$(\psi^g)_1=(\psi_1)^{g_1}\, .$$
  In particular, $(G_1)_{\lambda_1}$ is mapped to $G_\lambda$ via $\kappa_1$.
  \item There is a normal cyclic subgroup $C$ of $G_1$ with $C\sbs Z_1$, and a faithful $\nu \in \irr C$ such that
  $\nu^{Z_1}=\lambda_1 \in \irr{Z_1}$.
  \item $(G_1, C, \nu)_{\c H}$ is an $\c H$-triple.
  \item If $U=(G_1)_{\lambda_1}$ and $V=(G_1)_\nu$, 
  then $ U=Z_1V$ and $C=Z_1 \cap V$. Also $V=\cent {G_1} C$ and $C\sbs \zent V$.
  \end{enumerate}
  \end{thm}
  \begin{proof}
  Let $n=|G|$ and let
  $\c H_n \leq  {\rm Gal}(\QQ(\xi_n)/\QQ)$, 
  where $\xi_n$ is a primitive $n$th root of unity
  as in Section \ref{Htriples}.
  Notice that $\c H_n$ acts on the characters
  of any subgroup (or quotient) of $G$. Let $\mathbb F=\QQ(\xi_n)^{\c H_n}$, 
  so that $\c H_n={\rm Gal}(\QQ(\xi_n)/\mathbb F)$. 
  The fact that $(G,Z,\lambda)_{\c H}$ is
  an $\c H$-triple means that $\lambda$ is semi-invariant in 
  $G$ over $\mathbb F$ in the sense of \cite{L}
  (see page 47, second paragraph of \cite{L}). 
  Apply Theorem A and Corollary B of \cite{L}.
  The conjugation
  part in (b), which we shall later need,
   is not mentioned in Corollary B of \cite{L},
   but in Theorem 7.12(7) of \cite{T2}. 
   \end{proof}

 What follows is essentially a 
deep result by A. Turull concerning the 
Clifford theory and action of $\c H$
over Glauberman correspondents.

 \begin{thm}[Turull]\label{turull}
 Suppose that $G$ is a finite $p$-solvable group.
 Suppose that $K$ is a normal $p'$-subgroup of $G$, and that
 $Q$ is a $p$-subgroup such that $KQ \nor G$.  Let $D=\cent KQ$.
 Let $C$
 be a normal subgroup of $G$ such that $C \sbs \zent{KQ}$. Let $\nu \in \irr C$
 and assume that $(G, C, \nu)_\c H$ is an $\c H$-triple.
 Let $\Delta=\irrp{KQ|\nu^{\c H}}$ and let $\Delta'=\irrp{DQ|\nu^ \c H}$.
 Then there is an $\c H$-equivariant 
  bijection 
  $$f\colon\bigcup_{\tau\in \Delta} \irr{G|\tau} \rightarrow
 \bigcup_{\tau' \in \Delta'} \irr{\norm GQ|\tau'}\, .$$
 \end{thm}
 
 \begin{proof}
 Write $C=C_{p'} \times C_{p}$ and $\nu=\nu_{p'} \times \nu_{p}$,
 where $\nu_{p'} \in \irr{C_{p'}}$, and $\nu_{p'} \in \irr{C_{p}}$.
 Notice that $C_{p'} \sbs D$ and $C_p \sbs Q$ by hypothesis.
 
Write ${\rm Irr}_Q(K)$ for the $Q$-invariant irreducible characters
of $K$, and for each $\theta \in {\rm Irr}_Q(K)$, let $\hat \theta \in \irr{D}$ be its
Glauberman correspondent. 

By Theorem 3.2 of \cite{T4}, for each $\theta \in {\rm Irr}_Q(K)$ 
 there is an $\c H$-equivariant bijection
 $$f_\theta\colon  \irr{G|\theta^\c H} \cap \irr{G|\nu_p^\c H}
 \rightarrow 
 \irr{\norm GQ |{\hat \theta}^\c H} \cap \irr{\norm GQ |\nu_p^\c H}$$
 satisfying a list of conditions.   
We can take $f_{\theta^\sigma}=f_\theta$ for every
 $\sigma \in \c H$ and  $f_{\theta^x}=f_\theta$ for every $x \in \norm GQ$.
 
 We define now $f(\chi)$ for $\chi \in \bigcup_{\tau \in \Delta} \irr{G|\tau}$.
 Suppose that $\chi$ lies over some $\tau \in \Delta$. 
 We have that $\tau_{C}=\tau(1)\nu^\sigma$ for some $\sigma \in \c H$, 
 by using that $C \sbs \zent{KQ}$.
 Also, $\tau_K=\theta \in {\rm Irr}_Q(K)$ since $\tau(1)$ has $p'$-degree 
 and $|KQ:K|$ is a power of $p$. Note that $\theta$ lies over $\nu_{p'}^\sigma$.
 Define
 $f(\chi):=f_\theta(\chi) \in \irr{\norm G Q | {\hat \theta}^{\c H}}\cap\irr{\norm G Q | \nu_p^{\c H}}$.
 
 Note that $f$ is well-defined. First, any other constituent of $\chi_{KQ}$ is 
$\tau^x$ for some $x \in \norm G Q$ and $f_{\theta^x}=f_\theta$.
By Theorem 3.2(1) of \cite{T4} $f_\theta(\chi)$ lies over ${\hat \theta}^\sigma$,
hence over $\nu_{p'}^\sigma$. In order to see that $f$ is  well-defined we need that 
$f_\theta(\chi)$ lies also over $\nu_p^\sigma$ and this in principle is not guaranteed by
Theorem 3.2 of \cite{T4} but by Theorem 10.1 of \cite{Tur17}.
Hence $f_\theta(\chi)$ lies over $\nu^\sigma$, and consequently
over some $\tau'\in \Delta'$.

The map $f$ is clearly surjective as
every element in  $\bigcup_{\tau' \in \Delta'} \irr{\norm GQ|\tau'}$ lies in
$ \irr{\norm GQ |\mu} \cap \irr{\norm GQ |\nu_p^\sigma}$ for some $\sigma \in \c H$ and $\mu \in \irr {D| \nu_{p'}^{\c H}}$.
Again using that $f_{\theta^x}=f_\theta$ for every  $x \in \norm G Q$  and every $\theta \in {\rm Irr}_Q(K)$
one can check that $f$ is injective. 
 
 Finally $f$ is $\c H$-equivariant as every $f_\theta$ is $\c H$-equivariant.
\end{proof}

\medskip

In contrast to the proof of the reduction theorem for the McKay conjecture in \cite{IMN}, we cannot work
with characters of $p'$-degree. We have to work instead with characters of relative $p'$-degree. Those are defined as follows.
For $N \nor G$ and $\theta \in \irr N$, we say that $\chi \in \irr{G|\theta}$ 
has {\bf relative $p'$-degree} with respect to $N$ (or to $\theta$) if
the ratio $\chi(1)/\theta(1)$ is not divisible by $p$.
We denote by $\irrpr{G|\theta}$ the set of irreducible relative $p'$-degree characters with respect to $N$.
The following are easy properties of relative $p'$-degree characters.

\begin{lem}\label{relp}
Suppose that $N \nor G$ and $\theta \in \irr N$. Let $P \in \syl pG$.
\begin{enumerate}[{\rm (a)}]

\item
If $\chi \in \irrpr{G|\theta}$, then $\chi_N$ has some $P$-invariant irreducible
constituent and any two of them are $\norm GP$-conjugate. These $P$-invariant
constituents extend to $NP$.

\item
Suppose that  $N \sbs M \nor G$, and let $\chi \in \irr{G|\eta}$, where $\eta \in \irr{M|\theta}$.
Then $\chi \in \irrpr{G|\theta}$ if, and only if, $\chi \in \irrpr{G|\eta}$ and $\eta \in \irrpr{M|\theta}$.
\end{enumerate}
\end{lem}

\begin{proof}
We can write $\chi_{PN}=a_1\delta_1 + \ldots + a_k\delta_k$, where $\delta_i \in \irr{PN}$ lies
over some $G$-conjugate of $\theta$. In particular, $\delta_i(1)/\theta(1)$ is an integer.
Since $\chi(1)/\theta(1)$ is not divisible by $p$, it follows that there is some $i$ such that
$p$ does not divide $\delta_i(1)/\theta(1)$. Since this number divides $|PN:N|$, it follows
that $(\delta_i)_N=\eta \in \irr N$ is $P$-invariant. Suppose that $\eta^g$ is another $P$-invariant irreducible
constituent, then $P, P^g \sbs G_{\eta^g}$, and by Sylow theory, we have that $\eta^g$ and $\eta$
are $\norm GP$-conjugate. In particular, $\eta^g=\eta^h$ for some $h \in \norm G P$,
then $\eta^h$ also extends to $(PN)^h=PN$.  The second part easily follows using that
$\chi(1)/\theta(1)=(\chi(1)/\eta(1))(\eta(1)/\theta(1))$.
\end{proof}

 Let $Z\nor G$ and  $\lambda \in \irr Z$.  We will denote by $\irrpr{G|\lambda^{\c H}}$ the subset  
 of relative $p'$-degree characters of $\irr{G|\lambda^{\c H}}$ (with respect to $Z$).
Recall that whenever $X \leq G$ contains $G_{\lambda^\c H}$, the induction of characters
 defines an $\c H$-equivariant bijection
 $$\irr{X|\lambda^{\c H}}\rightarrow \irr{G|\lambda^{\c H}} \, .$$

 We are finally ready to prove the main result of this note.

\begin{thm}\label{reduction}
Let $Z \nor G$, $P \in \syl p G$
 and $\lambda \in \irr Z$ be $P$-invariant. Write $H=\norm G P Z$. Assume that every simple group involved in $G/Z$ satisfies the inductive Galois--McKay condition for $p$. Then there exists an $\c H$-equivariant bijection between 
  $\irrpr{G|\lambda^{\c H}}$ and $\irrpr{H|\lambda^{\c H}}$. 
 \end{thm}

\begin{proof}
 We argue by induction 
 on $|G:Z|$.
We also may assume that $H<G$, 
otherwise the statement trivially holds. 

\medskip

{\em Step 1.~~We may assume that $G=G_{\lambda^\c H}$. In particular $G_\lambda \nor G$, $G=G_\lambda H$ and $Z<G_\lambda$.}

Induction of characters defines
$\c H$-equivariant bijections $\irr{G_{\lambda^\c H}|\lambda^{\c H}}
\rightarrow \irr{G|\lambda^{\c H}}$ and $\irr{H_{\lambda^\c H}|\lambda^{\c H}}
\rightarrow \irr{H|\lambda^{\c H}}$. Since $P\sbs G_\lambda$, relative $p'$-degree characters
are mapped onto relative $p'$-degree characters. Hence we may assume $G=G_{\lambda^{\c H}}$.
In particular $(G,Z,\lambda)_{\c H}$ and
$(H,Z,\lambda)_{\c H}$ are $\c H$-triples,
and $G_\lambda \nor G$.
Using the Frattini argument we have that $G=G_\lambda H$.
If $G_\lambda=Z$, then $G=H$ contradicts our first assumption.

\medskip

{\em Step 2.~~We may assume that $G$ has a normal cyclic subgroup 
$C$ contained in $Z$ and a faithful character $\nu \in \irr C$ 
such that $\nu^Z=\lambda$ and $(G, C, \nu)_{\c H}$ is an
$\c H$-character triple. In particular, $G_\lambda=G_\nu Z$ and
$G_\nu \cap Z=C$. Moreover $G_\nu \nor G$ and $C\sbs \zent{G_\nu}$.}

By Theorem \ref{ladisch} there exists an $\c H$-character triple 
$(G_1, Z_1, \lambda_1)_{\c H}$, a group epimorphism 
$\kappa_1\colon G_1\rightarrow G/Z$ with kernel $Z_1$ and 
$\c H$-equivariant character bijections 
$\irrp{G|\lambda^{\c H}}\rightarrow \irrp{G_1|\lambda_1^{\c H}}$ 
and $\irrp{H|\lambda^{\c H}}\rightarrow \irrp{H_1|\lambda_1^{\c H}}$, 
where $\kappa_1(H_1)=H$ with $Z_1\sbs H_1$. 
These bijections also commute with group conjugation
as in Theorem \ref{ladisch}(b).
Let $Z_1\sbs (PZ)_1\leq (G_1)_{\lambda_1}$ be such that 
$\kappa_1((PZ)_1)=PZ$ and let $P_1 \in \syl p {(PZ)_1}$. 
Then $P_1\in \syl p {G_1}$, $(PZ)_1=P_1Z_1$ and also 
$H_1=\norm {G_1}{P_1}Z_1$ by the Frattini argument. 
By Theorem \ref{ladisch},
all the requirements of the claim are satisfied in $G_1$. 
Since $|G:Z|=|G_1: Z_1|$,
it is no loss if we work in $G_1$ instead of in $G$.

\medskip

{\em Step 3.~~If $H \sbs X< G$, then there
exists an $\c H$-equivariant bijection between 
  $\irrpr{X|\lambda^{\c H}}$ and $\irrpr{H|\lambda^{\c H}}$.}
  
  This follows by induction since $|X:Z|<|G:Z|$.
  
  \medskip

{\em Step 4.~~Let $L/Z$ be a chief factor
of $G$ with $L \sbs G_\lambda$. Then $G=LH$. In other words, $LP \nor G$.}

Recall that $H=\norm G P Z$.  
Define $\Theta_0$ to be
a complete set of representatives of the orbits
of $\norm GP \times \c H$ on the
$P$-invariant characters in
$\irrpr{L|\lambda^{\c H}}$.
By Lemma \ref{relp} every relative $p'$-degree character of $G$ with respect to $Z$ lies
over a unique $\norm GP$-orbit of
$P$-invariant characters  of relative $p'$-degree of $L$ with respect to $Z$.
Since $(G,Z, \lambda)_\c H$
is an $\c H$-triple, one can easily check that
$$\irrpr{G|\lambda^{\c H}}=\dot{ \bigcup_{\theta \in \Theta_0}} \irrpr{G|\theta^{\c H}}$$ 
is  a disjoint union. Similarly
$$\irrpr{LH|\lambda^{\c H}}= \dot{\bigcup_{\theta \in \Theta_0}} \irrpr{LH|\theta^{\c H}}.$$

Since $Z<L$, then  $|G: L|<|G  : Z|$, and
by induction 
we have $\c H$-equivariant bijections
$\irrpr{G|\theta^{\c H}} \rightarrow \irrpr{LH|\theta^{\c H}}$ whenever $\theta \in \Theta_0$.
This defines an $\c H$-equivariant bijection
$$\irrpr{G|\lambda^{\c H}} \rightarrow \irrpr{LH|\lambda^{\c H}}\,.$$
If $LH<G$, then by Step 3, we are done. The latter claim of the step
follows immediately since $\norm {G/L}{PL/L}=\norm G P L /L$
by the Frattini argument.

\medskip

{\em Step 5.~~We may assume $L/Z$ is not a $p$-group.}

Notice that, by Step 4, $G=LH$ where $H=\norm G PZ$. 
If $L/Z$ is a $p$-group then $H=G$, contradicting a previous assumption.

\medskip

{\em Step 6. Write $P_\nu=P\cap G_\nu$ and $K=L\cap G_\nu$. 
Then $H\sbs \norm G {P_\nu}<G$, $KP_\nu \nor G$
and $PZ=P_\nu Z$.}

Since $(G, C, \nu)_\c H$ is an $\c H$-triple, recall that $G_\nu\nor G$.
Then $K=L\cap G_\nu \nor G$.
Note that $(LP)_\nu=KP_\nu$.
Using that $ZP/Z$ is a Sylow $p$-subgroup of $G_\lambda/Z$,
that $P_\nu$ is a Sylow $p$-subgroup of $G_\nu$, and that
$G_\lambda=G_\nu Z$ with $G_\nu \cap Z=C$,
we conclude that $ZP_\nu=ZP$. By Dedekind's lemma
 $ZP \cap G_\nu=CP_\nu$. 
 A similar argument can be used to show that
$LP \cap G_\nu=KP_\nu$. (Note that $L=KZ$ with $K\cap Z=C$
and work with $Q=P_\nu\cap K \in \syl p K$.)

Recall that $LP\nor G$, hence and $KP_\nu=LP\cap G_\nu\nor G$.
Also $CP_\nu=C_{p'} \times P_\nu$, since $C$ is central in $G_\nu$.
Since $Z$ normalizes $CP_\nu$, it follows that $Z$ normalizes $P_\nu$.
Notice that if $P_\nu\nor G$, then $ZP \nor G$ (because $ZP_\nu =ZP$), a contradiction.
 Hence $H\sbs\norm G{P_\nu}<G$.

\medskip

{\em Step 7.~~Let $Y \leq G_\nu$ be such that $KY, LY \nor G$. Then
$$\irrpr{G|\lambda^\c H}=\bigcup_{\theta \in \Delta_Y} \irrpr{G|\theta^{LY}} =\bigcup_{\theta \in \Delta_Y} \irrpr{G|\theta} \, $$
where $\Delta_Y=\irrp{KY|\nu^{\c H}}$. Whenever $H\sbs X \leq G$ and $Y\sbs X$, we also have that
$$\irrpr{X|\lambda^\c H}=\bigcup_{\theta' \in \Delta'_Y} \irrpr{X|\theta^{LY\cap X}} =\bigcup_{\theta' \in \Delta'_Y} \irrpr{X|\theta} \, $$
where $\Delta'_Y=\irrp{KY\cap X|\nu^{\c H}}$.}

First note that $LY=(KY)Z$ and $KY\cap Z=C$. Moreover $(LY)_{\nu^\sigma}=KY$, 
whenever $\sigma \in \c H$. Note that induction of characters 
defines a bijection $\irr{KY|\nu^\sigma}\rightarrow \irr{LY | \lambda^\sigma}$
for every $\sigma \in \c H$.
Let $\chi \in \irrpr{G|\lambda^\c H}$.
Let $\mu \in \irr{LY}$ be under $\chi$ and over $\lambda^\sigma$, for some $\sigma \in \c H$. 
By Lemma \ref{relp}(a) $\mu \in \irrpr{LY|\lambda^\sigma}$ and $\chi(1)/\mu(1)$ is a $p'$-number. 
Let $\theta \in \irr{KY}$ be under $\mu$ and over $\nu^\sigma$. Then $\theta^{LY}=\mu$. 
Also $\mu(1)/\lambda(1)=\theta(1)/\nu(1)=\theta(1)$ is a $p'$-number.
The other inclusion is shown similarly. 
If $H \sbs X \le G$ and $Y\sbs X$, we can use the same argument to show that
$$\irrpr{X|\lambda^\c H}=\bigcup_{\theta' \in \Delta'_Y} \irrpr{X|(\theta')^{LY\cap X}}=\bigcup_{\theta' \in \Delta'_Y} \irrpr{X|\theta'} \, ,$$
where $\Delta'_Y=\irrp{ (LY \cap X)_\nu|\nu^\c H}$. Notice that in $X$ we still have 
$X=X_\lambda H$, 
$X_\lambda=X_\nu Z$ and the rest of the conditions with respect to the normal subgroups $LY \cap X=(L\cap X)Y$ and $KY\cap X=(K\cap X)Y$.
We have intentionally omitted the dependence of $\Delta'_Y$ on $X$ in the notation, but this shall not lead to confusion as the subgroup $X$ will be clear when applied this step. 
 
\medskip

{\em Step 8.~~We may assume $L/Z$ is not a $p'$-group.}

If $L/Z$ is a $p'$-group, we see that
$G$ is $p$-solvable and $K/C$ is a $p'$-group. Since $C \sbs \zent K$,
write $K=C_p \times K_{p'}$, where $K_{p'}$ is a normal 
$p$-complement of $K$.
Write $P_\nu=P\cap G_\nu$ and $X=\norm G{P_\nu}$. By Step 6,
$H\sbs X <G$. By Step 7, taking $Y=P_\nu$ we have
$$\irrpr{G|\lambda^\c H}=\bigcup_{\theta \in \Delta} \irrpr{G|\theta} =\bigcup_{\theta \in \Delta} \irr{G|\theta}\, $$
where $\Delta=\Delta_Y$ and we are using that $LP_\nu=LP$ and $|G:LP|$ is a $p'$-number. Also 
by Step 7, since $Y=P_\nu\sbs X$ we have
$$\irrpr{X|\lambda^\c H}=\bigcup_{\theta' \in \Delta'} \irr{X|\theta'} \, ,$$
where $\Delta'=\Delta'_Y$. Just note that 
$(LY\cap X)_\nu=KP_\nu\cap\norm{G_\nu}{P_\nu}=\norm K{P_\nu}P_\nu=DP_\nu$
where $D=\cent{K_{p'}} {P_\nu}=\norm {K_{p'}}{P_\nu}$.
Hence $\Delta=\irrp{K_{p'}P_\nu|\nu^\c H}$ and $\Delta'=\irrp{DP_\nu|\nu^\c H}$.
By Theorem \ref{turull}, there is an $\c H$-equivariant bijection
 $$f \colon \bigcup_{\theta \in \Delta} \irr{G|\theta} \rightarrow
 \bigcup_{\theta' \in \Delta'} \irr{X|\theta'} \, .$$
By applying Step 3 we are done in this case.

\medskip

{\em Final Step.}~~By Step 5 and Step 8, we may assume that $L/Z$ is a direct product
of non-abelian simple groups of order divisible by $p$ isomorphic to some $S$. Since $K/C\cong L/Z$
 and $C \sbs \zent K$ by Step 2, 
we have $C=\zent K$. 
Let $Q=P\cap K \in \syl p K$. Thus $\norm GP \sbs \norm GQ$.
Furthermore, since $K/C$ and $Z/C$ are normal subgroups of $G/C$ with $K\cap Z=C$,
we have that $Z$ normalizes $CQ=C_{p'} \times Q$. Thus $Z$ normalizes $Q$.
(Note that $\norm G Q <G$ because $\norm K Q <K$.)
Hence $|\norm GQ: Z|<|G:Z|$.

Since $S$ satisfies the inductive Galois--McKay condition,
by Corollary \ref{centralextension} there exist an $\norm G Q$-stable subgroup 
$\norm{K}{Q}\sbs M<K$ and an $\norm G Q\times \c H$-equivariant bijection
$$\Omega \colon \irrp{K|\nu^{\c H}} \rightarrow \irrp{M |\nu^{\c H}}$$
 such that, for every $\theta \in \irrp{K|\nu^{\c H}}$, there is a
 character-degree-ratio preserving $\c H$-equivariant bijection
$$\hat \tau_\theta \colon \irr{G |\theta^{\c H}} \rightarrow \irr{M\norm G Q|\varphi^{\c H}} \, ,$$
where $\varphi=\Omega(\theta)$.
Write $U=M\norm G  Q$ and notice that $H\sbs U<G$ and $U\cap K=M$. 
Moreover, $\hat \tau_{\theta^u}=\hat \tau_\theta$ and $\hat \tau_{\theta^\sigma}=\hat \tau_\theta$ 
for every $\theta \in \irrp K$, $u \in U$ and $\sigma \in \c H$.

\smallskip

By Step 3, there is an $\c H$-equivariant bijection
$$\irrpr{U|\lambda^{\c H}} \rightarrow \irrpr{H |\lambda^\c H} \, .$$
Hence, we only need to construct an $\c H$-equivariant bijection 
$$F\colon  \irrpr{G|\lambda^{\c H}} \rightarrow \irrpr{U|\lambda^{\c H}} \, .$$
By Step 7 we have
 $$\irrpr{G|\lambda^{\c H}}=\bigcup_{\theta \in \Delta} \irrpr{G|\theta} \text{ \ \ and \ \ } \irrpr{U|\lambda^{\c H}}=\bigcup_{\theta' \in \Delta'} \irrpr{U|\theta'}\, ,$$
where $\Delta= \irrp{K|\nu^{\c H}}$ and $\Delta'= \irrp{M|\nu^{\c H}}$. 
We can define $F$  as follows.
If $\chi \in \irrpr{G|\lambda^\c H}$, then $\chi \in \irrpr{G|\theta}$
for some $\theta  \in \irrp{K|\nu^{\sigma}}$ and $\sigma \in \c H$. Such $\theta$ is determined up to $U$-conjugacy.
Define $F(\chi)=\hat \tau_\theta(\chi) \in \irr{U|\varphi}$, where $\varphi=\Omega(\theta)\in \irrp{M|\nu^\sigma}$. (This latter
fact follows from the fact that $\chi$ and $\hat \tau_\theta(\chi)$ lie over the same character of $\zent K =C$ by the $\c H$-triples relations in Corollary
\ref{centralextension}.)

Now, $F$ is well-defined since $\hat \tau_{\theta}=\hat \tau_{\theta^u}$ 
for every $u \in U$. Suppose that $\chi, \chi' \in \irrpr{G |\lambda^{\c H}}$ have
the same image $\xi$ under $F$. If $\theta, \theta ' \in \irr K$  
lie under $\chi$ and $\chi'$ respectively,
then they must be $\norm G Q$-conjugate because $\Omega$ is 
$\norm G Q$-equivariant and $\Omega(\theta)$ and $\Omega(\theta')$ 
lie under $\xi$. Hence injectivity
also follows from the fact that
$\hat \tau_{\theta}=\hat \tau_{\theta^u}$ for every $u \in U$. 
$F$ is clearly surjective, and the proof is finished. 
\end{proof}

Theorem A follows from Theorem \ref{reduction} by taking $Z=1$.

\end{document}